\documentclass[a4paper,twoside,10pt]{article} %{report}
\newif\ifpdf
\ifx\pdfoutput\undefined
	\pdffalse              %%normal LaTeX is executed
\else
	\pdfoutput=1           
	\pdftrue               %%pdfLaTeX is executed
\fi

\ifpdf
\else
\fi

\usepackage[T1]{fontenc}
\usepackage[latin1]{inputenc}

\ifpdf %%Inclusion of graphics via \includegraphics{file}
	\usepackage[pdftex]{graphicx} %%graphics in pdfLaTeX
\else
	\usepackage[dvips]{graphicx} %%graphics and normal LaTeX
\fi
\usepackage{amsmath}
\usepackage{amsthm}
\usepackage{amsfonts}
\usepackage{amscd}

\usepackage{a4wide} %%Smaller margins = more text per page.
\usepackage{fancyhdr} %%Fancy headings
\begin{document}

%% File Extensions of Graphics %%%%%%%%%%%%%%%%%%%%%%%%%%%%%%
%% ==> This enables you to omit the file extension of a graphic.
%% ==> "\includegraphics{title.eps}" becomes "\includegraphics{title}".
%% ==> If you create 2 graphics with same content (but different file types)
%% ==> "title.eps" and "title.pdf", only the file processable by
%% ==> your compiler will be used.
%% ==> pdfLaTeX uses "title.pdf". LaTeX uses "title.eps".
\ifpdf
	\DeclareGraphicsExtensions{.pdf,.jpg,.png}
\else
	\DeclareGraphicsExtensions{.eps}
\fi

\pagestyle{headings}

%% Title Page %%%%%%%%%%%%%%%%%%%%%%%%%%%%%%%%%%%%%%%%%%%%%%%
%% ==> Write your text here or include other files.

%% The simple version:
\title{Strictly Singular Uniform $\lambda-$Adjustment in Banach Spaces}
\author{Boris Burshteyn}
\date{February, 2009} %%If commented, the current date is used.
\maketitle
\begin{abstract}
Based on the recently introduced uniform $\lambda-$adjustment for closed subspaces of Banach spaces we extend the concept of the strictly singular and finitely strictly singular operators to the sequences of closed subspaces and operators in Banach spaces and prove theorems about lower semi--Fredholm stability. We also state some new open questions related to strict singularity and the geometry of Banach spaces.  
\end{abstract}

%% The nice version:
%\input{titlepage} %%You need a file 'titlepage.tex' for this.
%% ==> TeXnicCenter supplies a possible titlepage file
%% ==> with its templates (File | New from Template...).

%% Inhaltsverzeichnis %%%%%%%%%%%%%%%%%%%%%%%%%%%%%%%%%%%%%%%
%\cleardoublepage %The first chapter should start on an odd page.

%\pagestyle{plain} %Now display headings: headings / fancy / ...
\pagestyle{headings} %Now display headings: headings / fancy / ...

%% Chapters %%%%%%%%%%%%%%%%%%%%%%%%%%%%%%%%%%%%%%%%%%%%%%%%%
%% ==> Write your text here or include other files.

\section[Introduction]{Introduction}\label{S:intro}
In the current paper we extend uniform $\lambda-$adjustment introduced recently in the author's work \cite{burshteyn} in order to encompass strictly singular and finitely strictly singular operators as well as strictly singular pairs of subspaces. 

In \cite{burshteyn} we defined the concept of \emph{uniform $\lambda-$adjustment} between sequences of subspaces of a Banach space that generalizes many of the previously known perturbations of closed subspaces and closed operators including perturbations by small gap, operator norm, $q-$norm, and $K_{2}-$approximation. It had been proved that perturbations with small $\lambda-$adjustment numbers preserve (semi--)Fredholm properties of linear operators as well as of pairs, tuples and complexes of closed subspaces. 

Strictly singular operators were first introduced by T. Kato in \cite{kato0}, strictly cosingular operators were first introduced by A. Pe\l czy\'nski in \cite{pelczynski} -- these concepts generalize compact operators and preserve (semi-)Fredholm properties of linear operators (see P. Aiena \cite{aiena} for an overview). Later M. Gonzales in \cite{gonzales} extended these concepts to the pairs of closed subspaces. A proper subclass of strictly singular operators -- finitely strictly singular operators -- had been defined and researched in V. D. Milman \cite{milman} and B. Sari, Th. Schlumprecht, N. Tomczak-Jaegermann, V. G. Troitsky \cite{Sari_Schlumprecht_Tomczak_Jaegermann_Troitsky}. 

The goal of the present work is to combine the ideas and results of strictly singular theory with uniform $\lambda-$adjustment. First, after recalling the concepts of uniform $\lambda-$adjustment and strict singularity, we define strictly singular and finitely strictly singular uniform $\lambda-$adjustment between sequences of closed subspaces. After that, we establish that the new concepts encompass strictly singular and finitely strictly singular linear operators as well as strictly singular pairs of subspaces, and that they are weaker than the previously considered uniform $\lambda-$adjustment. Then we prove the lower semi--Fredholm stability for sequences of pairs of subspaces, as well as for sequences of operators. In conclusion we discuss relaxed strict singularity, as well as a  relation of strict singularity to the geometry of Banach spaces.

\section{Acknowledgements}
The author is grateful to M.I. Ostrovskii for the provided reference to the work of A. Pe\l czy\'nski \cite{pelczynski2} related to the geometry of Banach spaces.

\subsection[Notational Conventions]{Notational Conventions}\label{S:conv}
$\mathbb{N}$ is a set of natural numbers, $\mathbb{N}^{'}$ is an infinite subset of $\mathbb{N}$, $\mathbb{N}^{''}$ is an infinite subset of $\mathbb{N}^{'}$, etc. When the number of nested subsets becomes high, we denote the subset at depth $n$ as $\mathbb{N}^{'n}$
\[
\mathbb{N}\ \supset\ \mathbb{N}^{'}\ \supset\ \mathbb{N}^{''}\ \supset\ \mathbb{N}^{'''}\ \supset\  \mathbb{N}^{'4}\ \supset\ \mathbb{N}^{'5}\ \supset\ ...
\]
A sequence of elements enumerated by elements ${n \in \mathbb{N}^{'}}$ is denoted by $(e_{n})_{\mathbb{N}^{'}}$.

If $(M_{n})_{\mathbb{N}^{'}}$ is a sequence of non-empty sets and $(x_{n})_{\mathbb{N}^{'}}$ is a sequence of elements such that $x_{n} \in M_{n}$ for $n \in \mathbb{N}^{'}$, then we say that $(x_{n})_{\mathbb{N}^{'}}\ is from\ (M_{n})_{\mathbb{N}^{'}}$ and write $(x_{n})_{\mathbb{N}^{'}} \triangleleft (M_{n})_{\mathbb{N}^{'}}$. 

If $(K_{n})_{\mathbb{N}^{'}}$ is a sequence of subsets such that $K_{n} \subset M_{n}$ for $n \in \mathbb{N}^{'}$, then we say that\linebreak $(K_{n})_{\mathbb{N}^{'}}\ is from\ (M_{n})_{\mathbb{N}^{'}}$ and write $(K_{n})_{\mathbb{N}^{'}} \prec (M_{n})_{\mathbb{N}^{'}}$.

A vector $x$ from a unit sphere of a Banach space $X$ is called \emph{a unit vector}; a sequence of unit vectors $(x_{n})_{\mathbb{N}^{'}}$ is called \emph{a unit sequence}.

All \emph{subspaces} and \emph{operators} in Banach spaces are meant to be \emph{linear}.

\emph{A null element} of a Banach space is denoted by $\theta$; \emph{a null subspace} of a Banach space -- the one that consists of a single element $\theta$ -- is denoted by $\{\theta\}$; \emph{a null operator} from a Banach space $X$ to a Banach space $Y$ -- the one that maps every vector from $X$ into $\theta$ from $Y$ -- is also denoted by $\theta$.

If ${X^{1},...,X^{k}}$ are ${k\geq2}$ Banach spaces, then their direct product ${\prod_{i=1}^{k}X^{i}\ =\ X^{1} \times \cdots \times X^{k}}$ is a Banach space of ordered $k-$tuples ${(x^{1},...,x^{k})}$ such that ${x^{i} \in X^{i}}$ for each ${i=1,...,k}$; the norm on ${\prod_{i=1}^{k}X^{i}}$ is defined as ${\max\{\left\|x^{1}\right\|,...,\left\|x^{k}\right\|\}}$.

If $X$ and $Y$ are two Banach spaces, then $\mathcal{C}(X,Y)$ is a set of \emph{closed linear operators} -- those which graphs are closed in the product space $X \times Y$; domain $dom(A)$ of a closed operator ${A \in \mathcal{C}(X,Y)}$ may be \emph{a proper subspace} of $X$. 

By ${\mathcal{B}(X,Y)}$ we denote a Banach space, furnished with operator norm, of \emph{continuous (i.e. bounded) operators} defined \emph{on all} $X$ and acting into $Y$; ${\mathcal{K}(X,Y)}$ is a space of all \emph{compact operators}. Note inclusions
\[
\mathcal{K}(X,Y)\ \subset\ \mathcal{B}(X,Y)\ \subset\ \mathcal{C}(X,Y).
\]
\emph{Dimension} of a Banach space $X$, denoted by ${\dim X}$, is the power of a maximal set of linearly independent vectors; if it is not finite, then we write ${\dim X = \infty}$. When dimension numbers are finite, their addition and subtraction  follow usual rules of arithmetics; when at least one of the dimensions is ${\infty}$, then by definition the result of any addition or subtraction is ${\infty}$ as well.

\newtheorem{theorem}{Theorem}[subsection]
\newtheorem{lemma}[theorem]{Lemma}
\newtheorem{proposition}[theorem]{Proposition}
\newtheorem{definition}[theorem]{Definition}
\newtheorem{example}[theorem]{Example}
\newtheorem{remark}[theorem]{Remark}

\section[Uniform $\lambda-$Adjustment and Strict Singularity]{Uniform $\lambda-$Adjustment and Strict Singularity}\label{S:ssula}

In this section we first recall the definitions of uniform $\lambda-$adjustment of subspaces and operators, strictly singular operators and pairs of subspaces as well as finitely strictly singular operators. Then we define the new concepts of strictly singular and finitely strictly singular uniform $\lambda-$adjustment of subspaces and operators. After that, we show that the latter concepts encompass the former ones and provide some examples of strictly singular and finitely strictly singular $\lambda-$adjustment.

\subsection[Uniform $\lambda-$Adjustment of Subspaces and Operators]{Uniform $\lambda-$Adjustment of Subspaces and Operators}\label{S:ula}

The following definitions of the uniform $\lambda-$adjustment had been first introduced in \cite{burshteyn}.

\begin{definition}[Uniform $\lambda-$Adjustment of Sequences of Subspaces]\label{D:ulass} 
Let $(M_{n})_{\mathbb{N}^{'}}$ and $(P_{n})_{\mathbb{N}^{'}}$ be a pair of sequences of closed subspaces from a Banach space $X$, ${M_{n} \neq \{\theta\}}$ for all ${n \in \mathbb{N}^{'}}$ and $\lambda \geq 0$. We say that $(M_{n})_{\mathbb{N}^{'}}$ is lower $uniformly\ \lambda-adjusted$ with $(P_{n})_{\mathbb{N}^{'}}$ if for any $\eta > 0$ and for any unit subsequence $(x_{n})_{\mathbb{N}^{''}}$ from $(M_{n})_{\mathbb{N}^{''}}$ there exists a subsequence $(y_{n})_{\mathbb{N}^{'''}}$ from $(P_{n})_{\mathbb{N}^{'''}}$ and a vector $z \in X$ such that 
\[
\varlimsup_{n \in \mathbb{N}^{'''}} \left\|x_{n} - y_{n} - z \right\|\ \leq\ \lambda + \eta. 
\]
The uniform $\lambda-$adjustment between $(M_{n})_{\mathbb{N}^{'}}$ and $(P_{n})_{\mathbb{N}^{'}}$ is a non-negative real number defined as
\[
\lambda_{\mathbb{N}^{'}}[M_{n}, P_{n}]\ :=\ \inf \{ \lambda \in \mathbb{R} \mid (M_{n})_{\mathbb{N}^{'}}\ \text{is lower uniformly $\lambda-$adjusted with}\ (P_{n})_{\mathbb{N}^{'}}\}.
\]
Also, by definition
\begin{multline}
\notag
\begin{aligned}
&\lambda_{\mathbb{N}^{'}}[M, P_{n}] &\ :=\ \lambda_{\mathbb{N}^{'}}[M_{n}, P_{n}]\ &\ where\ \ M = M_{n}\ for\ n \in {N}^{'},\\
&\lambda_{\mathbb{N}^{'}}[M_{n}, P] &\ :=\ \lambda_{\mathbb{N}^{'}}[M_{n}, P_{n}]\ &\ where\ \ P = P_{n}\ for\ n \in {N}^{'},\\
&\lambda[M, P]                      &\ :=\ \lambda_{\mathbb{N}^{'}}[M_{n}, P_{n}]\ &\ where\ \ P = P_{n}\ and\ M = M_{n}\ for\ n \in {N}^{'}.
\end{aligned}
\end{multline}
\end{definition}

\begin{definition}[Uniform $\lambda-$Adjustment of Sequences of Closed Linear Operators]\label{D:ulasclo0} 
Let ${X}$ and ${Y}$ are two Banach spaces. Consider $(A_{n})_{\mathbb{N}^{'}}$ and $(
B_{n})_{\mathbb{N}^{'}}$ -- a pair of sequences of operators from $\mathcal{C}(X,Y)$, as well as $(G_{A_{n}})_{\mathbb{N}^{'}}$ and $(G_{B_{n}})_{\mathbb{N}^{'}}$ -- sequences of their respective graphs in the product space ${X \times Y}$ (recall that a graph of a closed operator $A \in \mathcal{C}(X,Y)$ is a closed subspace of ${X \times Y}$ defined as set of ordered pairs ${\{(x, Ax) \mid x \in Dom(A) \subset X\}}$). 

We say that $(A_{n})_{\mathbb{N}^{'}}$ is lower $uniformly\ \lambda-adjusted$ with $(B_{n})_{\mathbb{N}^{'}}$ if the sequence of graphs $(G_{A_{n}})_{\mathbb{N}^{'}}$ is lower $uniformly\ \lambda-adjusted$ with the sequence of graphs $(G_{B_{n}})_{\mathbb{N}^{'}}$ in the product space ${X \times Y}$.

The uniform $\lambda-$adjustment between $(A_{n})_{\mathbb{N}^{'}}$ and $(B_{n})_{\mathbb{N}^{'}}$ is defined as uniform $\lambda-$adjustment between their sequences of graphs:
\[
\lambda_{\mathbb{N}^{'}}[A_{n}, B_{n}]\ :=\ \lambda_{\mathbb{N}^{'}}[G_{A_{n}}, G_{B_{n}}].
\]
Also, by definition
\begin{multline}
\notag
\begin{aligned}
&\lambda_{\mathbb{N}^{'}}[A, B_{n}] &\ :=\ \lambda_{\mathbb{N}^{'}}[A_{n}, B_{n}]\ &\ where\ \ A = A_{n}\ for\ n \in {N}^{'},\\
&\lambda_{\mathbb{N}^{'}}[A_{n}, B] &\ :=\ \lambda_{\mathbb{N}^{'}}[A_{n}, B_{n}]\ &\ where\ \ B = B_{n}\ for\ n \in {N}^{'},\\
&\lambda[A, B]                      &\ :=\ \lambda_{\mathbb{N}^{'}}[A_{n}, B_{n}]\ &\ where\ \ A = A_{n}\ and\ B = B_{n}\ for\ n \in {N}^{'}.
\end{aligned}
\end{multline}
\end{definition}

For detailed discussion on $\lambda-$adjustment see \cite{burshteyn} where it is proved that perturbations by small gap, small Hausdorf measure of non-compactness, small norm, as well as $K_{2}-$approximation are all particular cases of $\lambda-$adjustment and that a variety of (semi-)Fredholm stability theorems remain true in the context of $\lambda-$adjustment.

\subsection[Strictly Singular and Finitely Strictly Singular Operators and Pairs of Subspaces]{Strictly Singular and Finitely Strictly Singular Operators and Pairs of Subspaces}\label{S:ussfssops}

The following definition of a strictly singular operator is due to T. Kato \cite{kato0}: 

\begin{definition}[Strictly Singular Operator]\label{D:sso0}
If $X,Y$ are two Banach spaces then operator $A \in \mathcal{B}(X,Y)$ is called a strictly singular operator if for every ${\epsilon > 0}$ and for every closed subspace $Z \subset X$ with ${\dim Z = \infty}$ there exists ${z \in Z}$ such that ${\left\|Tz\right\| < \varepsilon \left\|z\right\|}$. In other words $A$'s restriction on any closed subspace $Z \subset X$ with ${\dim Z = \infty}$ is not an isomorphism between $Z$ and $A(Z)$. The set of strictly singular operators from $\mathcal{B}(X,Y)$ is denoted as $\mathcal{SS}(X,Y)$.
\end{definition}
In \cite{kato0} T. Kato had proved Fredholm stability theorem for perturbations by strictly singular operators. Also, it is well known that
\[
\mathcal{K}(X,Y)\ \subset\ \mathcal{SS}(X,Y)\ \subset\ \mathcal{B}(X,Y)
\]
and that in general each inclusion is proper (see P. Aiena \cite{aiena}, S. Goldberg, E. Thorp \cite{goldberg_thorp} and S. Goldberg \cite{goldberg0}.

The next version of strict singularity for subspaces had first appeared in M. Gonzales \cite{gonzales}:

\begin{definition}[Strictly Singular Pair of Subspaces]\label{D:fssps}
Let $M,N$ be a pair of closed subspaces from a Banach space $X$. Denote $Q_{N}$ to be a canonical surjection from $X$ onto a quotient space $X/N$, also denote $J_{M}$ to be an identity injection from $M$ into $X$. We say that the pair $(M,N)$ belongs to the class of strictly singular operators if ${Q_{N} \circ J_{M} \in \mathcal{S}\mathcal{S}(M, X/N)}$.
\end{definition}
In \cite{gonzales} M. Gonzales proved (semi-)Fredholm stability theorems for pairs belonging to the class of strictly singular operators.

The next refinement of strict singularity -- finitely strictly singular operators -- is due to V. D. Milman \cite{milman} and B. Sari, Th. Schlumprecht, N. Tomczak-Jaegermann, V. G. Troitsky \cite{Sari_Schlumprecht_Tomczak_Jaegermann_Troitsky}:
\begin{definition}[Finitely Strictly Singular Operator]\label{D:fsso}
If $X,Y$ are two Banach spaces then operator ${A \in \mathcal{B}(X,Y)}$ is called a finitely strictly singular operator if for every ${\varepsilon > 0}$ there exists ${n \in \mathbb{N}}$ such that for every subspace ${Z \subset X}$ with ${\dim Z \geq n}$ there exists ${z \in Z}$ such that ${\left\|Az\right\| < \varepsilon \left\|z\right\|}$. The set of all finitely strictly singular operators from ${\mathcal{B}(X,Y)}$ is denoted as ${\mathcal{FSS}(X,Y)}$.
\end{definition}
It is well known (see \cite{Sari_Schlumprecht_Tomczak_Jaegermann_Troitsky}) that 
\[
\mathcal{K}(X,Y)\ \subset\ \mathcal{FSS}(X,Y)\ \subset\ \mathcal{SS}(X,Y)
\]
and that in general inclusions are proper (see also V. D. Milman \cite{milman}, A. Plichko \cite{plichko}).
\subsection[Strictly Singular Uniform $\lambda-$Adjustment for Sequences of Subspaces]{Strictly Singular Uniform $\lambda-$Adjustment for Sequences of Subspaces}\label{S:ssulass}
Having just recalled in the previous subsection the existing concepts of uniform $\lambda-$adjustment and strict singularity, we combine them together in the following two definitions.
\begin{definition}[Strictly Singular Uniform $\lambda-$Adjustment]\label{ssua}
Let $(M_{n})_{\mathbb{N}^{'}}$ and $(P_{n})_{\mathbb{N}^{'}}$ be a pair of sequences of closed subspaces from a Banach space $X$, ${M_{n} \neq \{\theta\}}$ for all ${n \in \mathbb{N}^{'}}$ and $\lambda \geq 0$. Then the following definitions are in order:
\begin{enumerate}
\item We say that $(M_{n})_{\mathbb{N}^{'}}$ is $lower$\ $strictly$\ $singular$\ $uniformly$\ $\lambda-adjusted$ with $(P_{n})_{\mathbb{N}^{'}}$ if for any subsequence of closed subspaces $(K_{n})_{\mathbb{N}^{''}} \prec (M_{n})_{\mathbb{N}^{''}}$ such that ${\dim K_{n} = \infty}$ for all ${n \in \mathbb{N''}}$ there exists a subsequence of closed subspaces $(L_{n})_{\mathbb{N}^{'''}} \prec (K_{n})_{\mathbb{N}^{'''}}$ such that ${\dim L_{n} = \infty}$ for all ${n \in \mathbb{N'''}}$ with the property ${\lambda_{\mathbb{N}^{'''}}[L_{n}, P_{n}] \leq \lambda}$. Let $\mathcal{S}\mathcal{S}\Lambda_{\mathbb{N}^{'}}[M_{n}, P_{n}]$ be the set of all such real numbers $\lambda$; then the $strictly$\ $singular$\ $uniform$\ $\lambda-$$adjustment$\ between $(M_{n})_{\mathbb{N}^{'}}$ and $(P_{n})_{\mathbb{N}^{'}}$ is a non-negative real number defined as
\[
\mathcal{S}\mathcal{S}\lambda_{\mathbb{N}^{'}}[M_{n}, P_{n}]\ :=\ \inf \{ \lambda \in \mathcal{S}\mathcal{S}\Lambda_{\mathbb{N}^{'}}[M_{n}, P_{n}] \}.
\]
\item We say that $(M_{n})_{\mathbb{N}^{'}}$ is $lower$\ $finitely$\ $strictly$\ $singular$\ $uniformly$\ $\lambda-adjusted$ with $(P_{n})_{\mathbb{N}^{'}}$ if for any subsequence of subspaces $(K_{n})_{\mathbb{N}^{''}} \prec (M_{n})_{\mathbb{N}^{''}}$ such that ${\dim K_{n} < \infty}$ for all ${n \in \mathbb{N''}}$ and ${\lim_{n \in \mathbb{N}^{''}} \dim K_{n} = \infty}$ there exists a subsequence of subspaces $(L_{n})_{\mathbb{N}^{'''}} \prec (K_{n})_{\mathbb{N}^{'''}}$ such that ${\lim_{n \in \mathbb{N}^{'''}} \dim L_{n} = \infty}$ with the property ${\lambda_{\mathbb{N}^{'''}}[L_{n}, P_{n}] \leq \lambda}$. Let $\mathcal{F}\mathcal{S}\mathcal{S}\Lambda_{\mathbb{N}^{'}}[M_{n}, P_{n}]$ be the set of all such real numbers $\lambda$; then the $finitely$\ $strictly$\ $singular$\ $uniform$\ $\lambda-$$adjustment$ between $(M_{n})_{\mathbb{N}^{'}}$ and $(P_{n})_{\mathbb{N}^{'}}$ is a non-negative real number defined as
\[
\mathcal{F}\mathcal{S}\mathcal{S}\lambda_{\mathbb{N}^{'}}[M_{n}, P_{n}]\ :=\ \inf \{ \lambda \in \mathcal{F}\mathcal{S}\mathcal{S}\Lambda_{\mathbb{N}^{'}}[M_{n}, P_{n}] \}.
\]
\end{enumerate}
\end{definition}
As it has been noted in \cite{burshteyn}, $\lambda_{\mathbb{N}^{'}}[M_{n}, P_{n}] \leq 1$ for any two sequences of subspaces $(M_{n})_{\mathbb{N}^{'}}$ and $(P_{n})_{\mathbb{N}^{'}}$. It is easy to see that both numbers ${\mathcal{S}\mathcal{S}\lambda_{\mathbb{N}^{'}}[M_{n}, P_{n}]}$ and ${\mathcal{F}\mathcal{S}\mathcal{S}\lambda_{\mathbb{N}^{'}}[M_{n}, P_{n}]}$ are well defined and do not exceed ${\lambda_{\mathbb{N}^{'}}[M_{n}, P_{n}]}$ -- to check that choose ${L_{n} = K_{n}}$ for all ${n \in \mathbb{N^{''}}}$ for every case. In summary, the following lemma is true:
\begin{lemma}[Strictly Singular Uniform $\lambda-$Adjustment is Well Defined]\label{L:ssulawd}
Let $(M_{n})_{\mathbb{N}^{'}}$ and $(P_{n})_{\mathbb{N}^{'}}$ be a pair of sequences of closed subspaces from a Banach space $X$, ${M_{n} \neq \{\theta\}}$ for all ${n \in \mathbb{N}^{'}}$. Then both numbers ${\mathcal{S}\mathcal{S}\lambda_{\mathbb{N}^{'}}[M_{n}, P_{n}]}$ and ${\mathcal{F}\mathcal{S}\mathcal{S}\lambda_{\mathbb{N}^{'}}[M_{n}, P_{n}]}$ are well defined and satisfy the following relations
\begin{multline}
\notag
\begin{aligned}
0\ &\leq\ \min(\mathcal{S}\mathcal{S}\lambda_{\mathbb{N}^{'}}[M_{n}, P_{n}],\  \mathcal{F}\mathcal{S}\mathcal{S}\lambda_{\mathbb{N}^{'}}[M_{n}, P_{n}])\\
   &\leq\ \max(\mathcal{S}\mathcal{S}\lambda_{\mathbb{N}^{'}}[M_{n}, P_{n}],\  \mathcal{F}\mathcal{S}\mathcal{S}\lambda_{\mathbb{N}^{'}}[M_{n}, P_{n}])\ \leq\ \lambda_{\mathbb{N}^{'}}[M_{n}, P_{n}]\ \leq\ 1.
\end{aligned}
\end{multline}
\end{lemma}
\begin{remark}\label{r:0}
It is obvious that for every sequence of finite-dimensional subspaces ${(R_{n})_{\mathbb{N^{'}}} \subset X}$ and for any two sequences of closed subspaces ${(M_{n})_{\mathbb{N^{'}}},(P_{n})_{\mathbb{N^{'}}} \subset X}$
\[
\mathcal{S}\mathcal{S}\lambda_{\mathbb{N}^{'}}[M_{n} + R_{n}, P_{n}]\            =\ \mathcal{S}\mathcal{S}\lambda_{\mathbb{N}^{'}}[M_{n}, P_{n}],\\ 
\]
Similarly, it is obvious that if ${\varlimsup \dim R_{n} < \infty}$ then ${\mathcal{F}\mathcal{S}\mathcal{S}\lambda_{\mathbb{N}^{'}}[R_{n}, \{\theta\}] = 0}$. Moreover, for the same ${(R_{n})_{\mathbb{N^{'}}}}$ and for any two sequences of closed subspaces ${(M_{n})_{\mathbb{N^{'}}},(P_{n})_{\mathbb{N^{'}}} \subset X}$
\[
\mathcal{F}\mathcal{S}\mathcal{S}\lambda_{\mathbb{N}^{'}}[M_{n} + R_{n}, P_{n}]\ =\ \mathcal{F}\mathcal{S}\mathcal{S}\lambda_{\mathbb{N}^{'}}[M_{n}, P_{n}].
\]
\end{remark}
\begin{example}\label{Ex:1}
In general ${\mathcal{S}\mathcal{S}\lambda_{\mathbb{N}^{'}}[M_{n}, P_{n}] \neq \lambda_{\mathbb{N}^{'}}[M_{n}, P_{n}]}$ and ${\mathcal{F}\mathcal{S}\mathcal{S}\lambda_{\mathbb{N}^{'}}[M_{n}, P_{n}] \neq \lambda_{\mathbb{N}^{'}}[M_{n}, P_{n}]}$. 
\end{example}
\begin{proof}
Consider a partition of the set of natural numbers ${\mathbb{N}}$ into an infinite sequence of disjoint non-empty finite subsets
\[
\mathbb{N} = \bigcup_{n \in \mathbb{N}} B_{n},\ \ B_{n} \neq \emptyset\ for\ all\ n \in \mathbb{N},\ \ B_{i} \bigcap B_{j} = \emptyset\ for\ i\ \neq j.
\]
Define closed subspaces ${M_{n} \subset l_{\infty}}$ consisting of sequences ${(\alpha^{n}_{i})_{i \in \mathbb{N}}}$ which components with indices from ${\mathbb{N} \backslash B_{n}}$ are null:
\[
M_{n} := \{ (\alpha^{n}_{i})_{i \in \mathbb{N}} \in l_{\infty} \mid i \notin B_{n} \Rightarrow \alpha^{n}_{i} = 0 \}.
\]
It had been shown in example $2.4.1$ from \cite{burshteyn} that ${\lambda_{\mathbb{N}}[M_{n}, \{\theta\}] = 1/2}$. At the same time, if all sets ${B_{n}}$ are finite, then ${\dim B_{n} = card B_{n}}$ where ${card B_{n}}$ is the number of elements in ${B_{n}}$, thus ${\mathcal{S}\mathcal{S}\lambda_{\mathbb{N}}[M_{n}, \{\theta\}] = 0}$. Further, if ${card B_{n} < K < \infty}$ for all $n$, then ${\mathcal{F}\mathcal{S}\mathcal{S}\lambda_{\mathbb{N}}[M_{n}, \{\theta\}] = 0}$.
\end{proof}
\begin{definition}[Notation for Single Subspaces]\label{D:nss}
The following definitions extend notation for cases when one or both sequences of subspaces degenerate to a single subspace:
\begin{itemize}
\item if ${M = M_{n}}$ for all ${n \in {N}^{'}}$ then
\begin{multline}
\notag
\begin{aligned}
\mathcal{F}\mathcal{S}\mathcal{S}\lambda_{\mathbb{N}^{'}}[M, P_{n}]\ :=\  \mathcal{F}\mathcal{S}\mathcal{S}\lambda_{\mathbb{N}^{'}}[M_{n}, P_{n}],\  
\mathcal{S}\mathcal{S}\lambda_{\mathbb{N}^{'}}[M, P_{n}]\            :=\  \mathcal{S}\mathcal{S}\lambda_{\mathbb{N}^{'}}[M_{n}, P_{n}];
\end{aligned}
\end{multline}
\item if ${P = P_{n}}$ for all ${n \in {N}^{'}}$ then
\begin{multline}
\notag
\begin{aligned}
\mathcal{F}\mathcal{S}\mathcal{S}\lambda_{\mathbb{N}^{'}}[M_{n}, P]\ :=\                         \mathcal{F}\mathcal{S}\mathcal{S}\lambda_{\mathbb{N}^{'}}[M_{n}, P_{n}],\ 
\mathcal{S}\mathcal{S}\lambda_{\mathbb{N}^{'}}[M_{n}, P]\            :=\ \mathcal{S}\mathcal{S}\lambda_{\mathbb{N}^{'}}[M_{n}, P_{n}];
\end{aligned}
\end{multline}
\item if ${M = M_{n}}$ and ${P = P_{n}}$ for all ${n \in {N}^{'}}$ then
\begin{multline}
\notag
\begin{aligned}
\mathcal{F}\mathcal{S}\mathcal{S}\lambda[M, P]\ :=\ \mathcal{F}\mathcal{S}\mathcal{S}\lambda_{\mathbb{N}^{'}}[M_{n}, P_{n}],\  
\mathcal{S}\mathcal{S}\lambda[M, P]\            :=\ \mathcal{S}\mathcal{S}\lambda_{\mathbb{N}^{'}}[M_{n}, P_{n}].
\end{aligned}
\end{multline}
\end{itemize}
\end{definition}
\subsection[Strictly Singular Uniform $\lambda-$Adjustment for Sequences of Operators]{Strictly Singular Uniform $\lambda-$Adjustment for Sequences of Operators}\label{S:ssulaso}
Our final two definitions induce strictly singular and finitely strictly singular $\lambda-$adjustment from closed subspaces to closed linear operators in a usual way by applying subspace concepts to the operator graphs in the product space.
\begin{definition}[Strictly Singular Uniform $\lambda-$Adjustment of Sequences of Closed Linear Operators]\label{D:ssulssclo} 
Let ${X}$ and ${Y}$ are two Banach spaces. Consider $(A_{n})_{\mathbb{N}^{'}}$ and $(
B_{n})_{\mathbb{N}^{'}}$ -- a pair of sequences of operators from $\mathcal{C}(X,Y)$, as well as $(G_{A_{n}})_{\mathbb{N}^{'}}$ and $(G_{B_{n}})_{\mathbb{N}^{'}}$ -- sequences of their respective graphs in the product space ${X \times Y}$. 
\begin{itemize}
\item We say that $(A_{n})_{\mathbb{N}^{'}}$ is $lower$ $strictly$ $singular$ $lower$ $uniformly$ $\lambda-$$adjusted$ with $(B_{n})_{\mathbb{N}^{'}}$ if the sequence of graphs $(G_{A_{n}})_{\mathbb{N}^{'}}$ is lower strictly singular uniformly $\lambda-$adjusted with the sequence of graphs $(G_{B_{n}})_{\mathbb{N}^{'}}$ in the product space ${X \times Y}$. The lower strictly singular uniform $\lambda-$adjustment between $(A_{n})_{\mathbb{N}^{'}}$ and $(B_{n})_{\mathbb{N}^{'}}$ is defined as lower strictly singular uniform $\lambda-$adjustment between their sequences of graphs:
\[
\mathcal{S}\mathcal{S}\lambda_{\mathbb{N}^{'}}[A_{n}, B_{n}]\ :=\ \mathcal{S}\mathcal{S}\lambda_{\mathbb{N}^{'}}[G_{A_{n}}, G_{B_{n}}].
\]
\item We say that $(A_{n})_{\mathbb{N}^{'}}$ is $lower$ $finitely$ $strictly$ $singular$ $lower$ $uniformly$ $\lambda-$$adjusted$ with $(B_{n})_{\mathbb{N}^{'}}$ if the sequence of graphs $(G_{A_{n}})_{\mathbb{N}^{'}}$ is lower finitely strictly singular uniformly $\lambda-$adjusted with the sequence of graphs $(G_{B_{n}})_{\mathbb{N}^{'}}$ in the product space ${X \times Y}$. The lower finitely strictly singular uniform $\lambda-$adjustment between $(A_{n})_{\mathbb{N}^{'}}$ and $(B_{n})_{\mathbb{N}^{'}}$ is defined as lower finitely strictly singular uniform $\lambda-$adjustment between their sequences of graphs:
\[
\mathcal{F}\mathcal{S}\mathcal{S}\lambda_{\mathbb{N}^{'}}[A_{n}, B_{n}]\ :=\ \mathcal{F}\mathcal{S}\mathcal{S}\lambda_{\mathbb{N}^{'}}[G_{A_{n}}, G_{B_{n}}].
\]
\end{itemize}
\end{definition}
\begin{definition}[Notation for Single Operators]\label{D:nso}
The following definitions extend notation for cases when one or both sequences of operators degenerate to a single operator:
\begin{itemize}
\item if ${A = A_{n}}$ for all ${n \in {N}^{'}}$ then
\begin{multline}
\notag
\begin{aligned}
\mathcal{F}\mathcal{S}\mathcal{S}\lambda_{\mathbb{N}^{'}}[A, B_{n}]\ :=\  \mathcal{F}\mathcal{S}\mathcal{S}\lambda_{\mathbb{N}^{'}}[A_{n}, B_{n}],\ 
\mathcal{S}\mathcal{S}\lambda_{\mathbb{N}^{'}}[A, B_{n}]\            :=\  \mathcal{S}\mathcal{S}\lambda_{\mathbb{N}^{'}}[A_{n}, B_{n}];
\end{aligned}
\end{multline}
\item if ${B = B_{n}}$ for all ${n \in {N}^{'}}$ then
\begin{multline}
\notag
\begin{aligned}
\mathcal{F}\mathcal{S}\mathcal{S}\lambda_{\mathbb{N}^{'}}[A_{n}, B]\ :=\ \mathcal{F}\mathcal{S}\mathcal{S}\lambda_{\mathbb{N}^{'}}[A_{n}, B_{n}],\ 
\mathcal{S}\mathcal{S}\lambda_{\mathbb{N}^{'}}[A_{n}, B]\            :=\ \mathcal{S}\mathcal{S}\lambda_{\mathbb{N}^{'}}[A_{n}, B_{n}];
\end{aligned}
\end{multline}
\item if ${A = A_{n}}$ and ${B = B_{n}}$ for all ${n \in {N}^{'}}$ then
\begin{multline}
\notag
\begin{aligned}
\mathcal{F}\mathcal{S}\mathcal{S}\lambda[A, B]\ :=\ \mathcal{F}\mathcal{S}\mathcal{S}\lambda_{\mathbb{N}^{'}}[A_{n}, B_{n}],\ 
\mathcal{S}\mathcal{S}\lambda[A, B]\            :=\ \mathcal{S}\mathcal{S}\lambda_{\mathbb{N}^{'}}[A_{n}, B_{n}].
\end{aligned}
\end{multline}
\end{itemize}
\end{definition}
We now establish relation between strict singularity of a bounded operator with its strictly singular $0-$adjustment with the null operator.
\begin{lemma}[Strict Singularity of Bounded Operators Implies Their Strictly Singular $0-$Adjustment with the Null Operator]\label{ssboiissa}
Let ${X,Y}$ be two Banach spaces and ${(A_{n})_{\mathbb{N^{'}}} \subset \mathcal{SS}(X,Y)}$. Then\linebreak ${\mathcal{S}\mathcal{S}\lambda_{\mathbb{N^{'}}}[A_{n}, \theta] = 0}$.
\end{lemma}
\begin{proof}
First recall the measure of strict singularity ${\Delta(A)}$ for a bounded operator ${A \in \mathcal{B}(X,Y)}$ introduced by M. Schechter in \cite{schechter}:
\[
\Delta(A) \ =\ \sup_{M}\ \inf_{N}\ \{\left\|A \mid _{N}\right\|\ \mid\ N \subset M \subset X,\ \dim N = \dim M = \infty\}.
\]
It has been proved in \cite{schechter} that operator $A$ is strictly singular if and only if ${\Delta(A) = 0}$. It is also clear by the straight application of the definitions that ${\mathcal{S}\mathcal{S}\lambda_{\mathbb{N^{'}}}[A_{n}, \theta] \leq \varlimsup_{n \in \mathbb{N^{'}}} \Delta(A_{n})}$ for any sequence of bounded operators ${(A_{n})_{\mathbb{N^{'}}} \subset \mathcal{B}(X,Y)}$. Therefore, if all $A_{n}$ are strictly singular then
\[
0\ \leq\ \mathcal{S}\mathcal{S}\lambda_{\mathbb{N^{'}}}[A_{n}, \theta]\ \leq\ \varlimsup_{n \in \mathbb{N^{'}}} \Delta(A_{n})\ =\ 0,
\]
thus ${\mathcal{S}\mathcal{S}\lambda[A, \theta] = 0}$.
\end{proof}

\begin{example}\label{Ex:2}[Strictly singular adjustment does not imply strict singularity]
When a sequence of operators is lower strictly singular $\lambda-$adjusted with the null operator, then individual operators from the sequence themselves do not need to be strictly singular.
\end{example}
\begin{proof}
Consider subspaces $M_{n}$ from $l_{\infty}$ from the previous Example \ref{Ex:1} and define a natural projection ${P_{n} : l_{\infty} \rightarrow M_{n}}$. By the reasoning from the Example 2.4.1 from \cite{burshteyn} it is clear that ${\lambda_{\mathbb{N}}[P_{n}, \{\theta\}] = \frac{1}{2}}$ yet when ${card B_{n} = \infty}$ then none of the projections $P_{n}$ is strictly singular -- in fact its Schechter measure of strict singularity ${\Delta(P_{n})}$ is equal to $1$.
\end{proof}
Now we show that strictly singular $0-$adjustment is weaker than the membership of a pair of subspaces in the class of strictly singular operators.
\begin{lemma}[A Pair of Subspaces from Strictly Singular Class is Strictly Singularly $0-$Adjusted]\label{pssscssa}
If $(M,N)$ is a pair of closed subspaces from a Banach space $X$ that belongs to the class of strictly singular operators (see Definition \ref{D:fssps}) then ${\mathcal{S}\mathcal{S}\lambda[M,N] = 0}$.
\end{lemma}
\begin{proof}
Let $Q_{N}$ be the canonical surjection from $X$ onto a quotient space $X/N$ and $J_{M}$ be the identity injection from $M$ into $X$. Then by the lemma's condition operator ${Q_{N} \circ J_{M} \in \mathcal{S}\mathcal{S}(M, X/N)}$. Therefore, according to the previous Lemma \ref{ssboiissa} ${\mathcal{S}\mathcal{S}\lambda[Q_{N} \circ J_{M}, \theta] = 0}$. By the straight application of definition of $\lambda-$adjustment to the graph of ${Q_{N} \circ J_{M}}$ it is easy to see that ${\mathcal{S}\mathcal{S}\lambda[M,N] = 0}$.
\end{proof}

\begin{remark}\label{r:1}
It is easy to see that if ${(C_{n})_{\mathbb{N}^{'}} \subset \mathcal{B}(X,Y)}$ is a sequence of finite rank operators, then for any pair of sequences of operators ${(A_{n})_{\mathbb{N}^{'}}, (B_{n})_{\mathbb{N}^{'}} \subset \mathcal{C}(X,Y)}$
\begin{multline}
\notag
\begin{aligned}
\mathcal{S}\mathcal{S}\lambda_{\mathbb{N}^{'}}[A_{n}, B_{n}]\            =\ \mathcal{S}\mathcal{S}\lambda_{\mathbb{N}^{'}}[A_{n} + C_{n}, B_{n}],\\ 
\end{aligned}
\end{multline}
If in addition the ranks of all ${C_{n}}$ are limited from above, then
\begin{multline}
\notag
\begin{aligned}
\mathcal{F}\mathcal{S}\mathcal{S}\lambda_{\mathbb{N}^{'}}[A_{n}, B_{n}]\ =\ \mathcal{F}\mathcal{S}\mathcal{S}\lambda_{\mathbb{N}^{'}}[A_{n} + C_{n}, B_{n}].
\end{aligned}
\end{multline}
Also if ${\left\|C_{n}\right\| \rightarrow 0}$, then 
\[
\mathcal{S}\mathcal{S}\lambda_{\mathbb{N}^{'}}[A_{n}, B_{n}]\            =\ \mathcal{S}\mathcal{S}\lambda_{\mathbb{N}^{'}}[A_{n} + C_{n}, B_{n}],\ \ 
\mathcal{F}\mathcal{S}\mathcal{S}\lambda_{\mathbb{N}^{'}}[A_{n}, B_{n}]\ =\ \mathcal{F}\mathcal{S}\mathcal{S}\lambda_{\mathbb{N}^{'}}[A_{n} + C_{n}, B_{n}].
\]
Also it is easy to see that if ${A \in \mathcal{B}(X, Y)}$, ${(B_{n})_{\mathbb{N}^{'}} \subset \mathcal{B}(X,Y)}$ and ${\mathcal{F}\mathcal{S}\mathcal{S}\lambda_{\mathbb{N}^{'}}[B_{n}, \theta] = 0},$\linebreak then ${\mathcal{F}\mathcal{S}\mathcal{S}\lambda_{\mathbb{N}^{'}}[A + B_{n}, A] = 0}$.
\end{remark}
Next we prove the following:
\begin{lemma}[Finitely Strict Singularity of a Bounded Operator Implies its Finitely Strictly Singular $0-$Adjustment with a Null Operator]\label{fssboiifssa}
Let ${X,Y}$ be two Banach spaces and ${A \in \mathcal{FSS}(X,Y)}$. Then ${\mathcal{F}\mathcal{S}\mathcal{S}\lambda[A, \theta] = 0}$.
\end{lemma}
\begin{proof}
Consider a number ${\eta > 0}$ and a subsequence of closed subspaces ${(K_{n})_{\mathbb{N}^{''}}}$ from ${G_{A}}$ such that ${\dim K_{n} < \infty}$ for all ${n \in \mathbb{N''}}$ and ${\lim_{n \in \mathbb{N}^{''}} \dim K_{n} = \infty}$ -- our goal is to find a subsequence of closed subspaces $(L_{n})_{\mathbb{N}^{'''}} \prec (K_{n})_{\mathbb{N}^{'''}}$ such that ${\lim_{n \in \mathbb{N}^{'''}} \dim L_{n} = \infty}$ with the property ${\lambda_{\mathbb{N}^{'''}}[L_{n}, \theta] \leq \eta}$. For that purpose define ${D_{n} = \{x \in X \mid (x, Ax) \in K_{n}\}}$ for each ${n \in \mathbb{N''}}$ and suppose that we have found a subsequence of subspaces $(R_{n})_{\mathbb{N}^{'''}} \prec (D_{n})_{\mathbb{N}^{'''}}$ such that ${\lim_{n \in \mathbb{N}^{'''}} \dim R_{n} = \infty}$ and that ${\left\|A\mid_{R_{n}}\right\| \leq \eta}$ for every ${n \in \mathbb{N}^{'''}}$. Then we can define ${L_{n}}$ to be a graph of the restriction of $A$ onto $R_{n}$ -- it is easy to see that ${\lambda_{\mathbb{N}^{'''}}[L_{n}, \theta] \leq \eta}$ since if ${(x_{n}, Ax_{n})_{\mathbb{N}^{'4}} \triangleleft (L_{n})_{\mathbb{N}^{'4}}}$ is a unit sequence then one can choose ${(y_{n})_{\mathbb{N}^{'4}} \subset G_{\theta}}$ as ${(x_{n}, \theta)_{\mathbb{N}^{'4}}}$ and ${z \in (X,Y)}$ as $(\theta, \theta)$ so that for every ${n \in \mathbb{N}^{'4}}$
\[
\left\|(x_{n}, Ax_{n}) - y_{n} - z\right\|\ =\ \left\|(x_{n}, Ax_{n}) - (x_{n}, \theta) - (\theta, \theta)\right\|\ =\ \left\|(\theta, Ax_{n})\right\|\ \leq\ \eta.
\]
The above inequality means that ${\lambda_{\mathbb{N}^{'''}}[L_{n}, \theta] \leq \eta}$. Since $\eta$ can be arbitrarily small, we must conclude that ${\mathcal{F}\mathcal{S}\mathcal{S}\lambda[A, \theta] = 0}$.

Now let us enumerate the elements from ${\mathbb{N}^{''} = \{k_{1} < k_{2} < ... < k_{n} < ...\}}$ and build the subsequence of numbers ${\{r_{1} < r_{2} < ... < r_{n} < ...\} = \mathbb{N}^{'''} \subset \mathbb{N}^{''}}$ and a sequence of needed subspaces ${(R_{n})_{\mathbb{N}^{'''}} \prec (D_{n})_{\mathbb{N}^{'''}}}$ by induction.

Since $A$ is finitely strictly singular, there exists a number ${r \in \mathbb{N}}$ such that for every subspace ${R \subset X}$ with ${\dim R \geq r}$ there exists a unit vector ${x_{r} \in R}$ such that ${\left\|Ax_{r}\right\| \leq \eta}$. Therefore, since ${\lim_{k \in \mathbb{N}^{''}} \dim D_{k} = \infty}$ we can find ${r_{1} = \min\{k \in \mathbb{N}^{''} \mid \dim D_{k} > r\}}$, then we can choose a unit vector ${x_{r_{1}} \in D_{r_{1}}}$ such that ${\left\|Ax_{r_{1}}\right\| \leq \eta}$ and construct a one-dimensional subspace ${R_{r_{1}} = sp\{x_{r_{1}}\} \subset D_{r_{1}}}$.

Suppose we have  constructed $j \geq 1$ numbers ${\{r_{1} < r_{2} < ... < r_{j}\} \subset \mathbb{N}^{''}}$ and $j$ subspaces ${R_{r_{i}} \subset D_{r_{i}}}$ such that ${\left\|A \mid _{R_{r_{i}}}\right\| \leq \eta}$ and ${\dim R_{r_{i}} = i}$ for ${i = 1, 2, ..., j}$. Our goal now is to find a new number ${r_{j+1} \in \mathbb{N}^{''}}$ and a new subspace ${R_{r_{j+1}} \subset D_{r_{j+1}}}$ such that ${r_{j+1} > r_{j}}$, ${\left\|A \mid _{R_{r_{j+1}}}\right\| \leq \eta}$ and\linebreak ${\dim R_{r_{j+1}} = j+1}$.

Since $A$ is finitely strictly singular, there exists a number ${r \in \mathbb{N}}$ such that for every subspace ${R \subset X}$ with ${\dim R \geq r}$ there exists a unit vector ${x_{r} \in R}$ such that ${\left\|Ax_{r}\right\| \leq \eta 2^{-(j+1)}}$. Therefore, since ${\lim_{k \in \mathbb{N}^{''}} \dim D_{k} = \infty}$ we can find ${r_{j+1} = \min\{k \in \mathbb{N}^{''} \mid \dim D_{k} > r + j + 1\}}$, then we can choose a unit vector ${x^{1}_{r_{j+1}} \in D_{r_{j+1}}}$ such that ${\left\|Ax^{1}_{r_{j+1}}\right\| \leq 2^{-(j+1)} \eta}$. By the Hahn-Banach theorem there exists a continuous unit functional ${f_{1} \in D_{r_{j+1}}^{*}}$ such that ${f_{1}x^{1}_{r_{j+1}} = 1}$. The kernel of that functional ${N_{1} = Ker f_{1}}$ has dimension ${\dim D_{k} - 1 > r + j > r}$. Therefore, we may find another unit vector ${x^{2}_{r_{j+1}} \in N_{1}}$ such that ${\left\|Ax^{2}_{r_{j+1}}\right\| \leq 2^{-(j+1)} \eta}$. As before, by the Hahn-Banach theorem there exists a continuous unit functional ${f_{2} \in N_{1}^{*}}$ such that ${f_{2}x^{2}_{r_{j+1}} = 1}$. The kernel of that functional ${N_{2} = Ker f_{2}}$ has dimension ${\dim D_{k} - 2 > r + j - 1}$. Clearly we can continue this at least $j+1$ times thus constructing a set of ${j+1}$ finite-dimensional subspaces ${D_{r_{j+1}} = N_{0} \supset N_{1} \supset ... \supset N_{j}}$ with the set of $j+1$ unit vectors ${x^{i}_{j+1} \in N_{i-1}}$ so that ${N_{i-1} = sp\{x_{i}\} \oplus N_{i}}$ for each ${i = 1,...,j+1}$.

Now define a projection ${P_{i} : N_{i-1} \rightarrow sp\{x_{i}\}}$ with the kernel ${N_{i}}$ for each ${i = 1,...,j+1}$. Finally define a new ${(j+1)-}$dimensional subspace ${R_{r_{j+1}} = sp\{x^{1}_{r_{j}},...,x^{j+1}_{r_{j}}\}}$. Now let us prove that ${\left\|Ax\right\| \leq \eta}$ for any unit vector ${x \in R_{r_{j+1}}}$. For that decompose ${x = \sum \alpha_{i} x^{i}_{r_{j}}}$ into the sum of its $j+1$ coordinates and calculate
\[
\left\|Ax\right\|\ =\ \left\|A(\sum_{i=1}^{j+1} \alpha_{i} x^{i}_{r_{j}})\right\|\ \leq\ \sum_{i=1}^{j+1} \left| \alpha_{i} \right| \left\|Ax^{i}_{r_{j}}\right\|\ \leq\ 2^{-(j+1)} \eta \times \sum_{i=1}^{j+1} \left| \alpha_{i} \right|.
\]
However, it is clear from our construction that ${\left| \alpha_{1} \right| = \left\| P_{1}x \right\| = 1}$ and\linebreak ${\left| \alpha_{i} \right| = \left\|P_{i} \circ (I_{i-1} - P_{i-1}) \circ ... \circ (I_{1} - P_{1}) x\right\|}$ where ${I_{i} : N_{i} \rightarrow N_{i}}$ is the identity map for each\linebreak ${i = 2,...,j+1}$. Therefore, for the same ${i}$
\begin{multline}
\notag
\begin{aligned}
\left| \alpha_{i} \right|\ &\leq\ \left\| P_{i} \right\| \times (\prod_{j=1}^{i-1} \left\| I_{j} - P_{j} \right\|) \times \left\| x \right\|\ =\ 1 \times (\prod_{j=1}^{i-1} \left\| I_{j} - P_{j} \right\|) \times 1\\
                           &\leq\ \prod_{j=1}^{i-1} (\left\| I_{j} \right\| + \left\| P_{j} \right\|)\ =\ \prod_{j=1}^{i-1} (1+1)\ =\ 2^{(i-1)} 
\end{aligned}
\end{multline}
since ${\left\|P_{i}\right\| = 1}$ for those $i$. Thus the previous estimate for ${\left\|Ax\right\|}$ can be further rewritten
\begin{multline}
\notag
\begin{aligned}
\left\|Ax\right\|\ \leq\ 2^{-(j+1)} \eta \times \sum_{i=1}^{j+1} \left| \alpha_{i} \right|\ \leq\ 2^{-(j+1)} \eta  \times \sum_{i=1}^{j+1} 2^{(i-1)}\ =\ \eta \times 2^{-(j+1)} \times (2^{j+1} - 1)\ <\ \eta
\end{aligned}
\end{multline}
and this concludes the entire proof.
\end{proof}

The next lemma shows that operator compositions supply an abundance of operator sequences that are  strictly singular $0-$adjusted with the null operator. This way strictly singular adjustment resembles the ideal properties of strictly singular operators (see P. Aiena \cite{aiena} and J. Diestel, H. Jarchow, A. Pietsch \cite{diestel_jaqrchow_pietsch}).

\begin{lemma}[Strict Singularity of Composition Operator Sequences]\label{ssco}
Let ${W}$, $X$, $Y$ and $Z$ are Banach spaces with three sequences of linear operators ${(S_{n})_{\mathbb{N}^{'}} \subset \mathcal{B}(Y, Z)}$, ${(A_{n})_{\mathbb{N}^{'}} \subset \mathcal{B}(X, Y)}$ and ${(T_{n})_{\mathbb{N}^{'}} \subset \mathcal{B}(W, X)}$:
\[
\begin{CD}
W @>T_{n}>> X @>A_{n}>> Y @>S_{n}>> Z.
\end{CD}
\]
Suppose that ${\varlimsup \left\|S_{n}\right\| = s < \infty}$ and ${\varlimsup \left\|T_{n}\right\| = t < \infty}$. Then the following propositions are true:
\begin{enumerate}
  \item If ${\mathcal{F}\mathcal{S}\mathcal{S}\lambda_{\mathbb{N}^{'}}[A_{n},\ \theta] = 0}$, then ${\mathcal{F}\mathcal{S}\mathcal{S}\lambda_{\mathbb{N}^{'}}[A_{n} \circ T_{n},\ \theta] = \mathcal{F}\mathcal{S}\mathcal{S}\lambda_{\mathbb{N}^{'}}[S_{n} \circ A_{n},\ \theta] = 0}$.
  \item If ${\mathcal{S}\mathcal{S}\lambda_{\mathbb{N}^{'}}[A_{n},\ \theta] = 0}$, then ${\mathcal{S}\mathcal{S}\lambda_{\mathbb{N}^{'}}[A_{n} \circ T_{n},\ \theta] = \mathcal{S}\mathcal{S}\lambda_{\mathbb{N}^{'}}[S_{n} \circ A_{n},\ \theta] = 0}$.
\end{enumerate}
\end{lemma}
\begin{proof}
Suppose that ${\mathcal{F}\mathcal{S}\mathcal{S}\lambda_{\mathbb{N}^{'}}[A_{n},\ \theta] = 0}$ and let us prove that ${\mathcal{F}\mathcal{S}\mathcal{S}\lambda_{\mathbb{N}^{'}}[A_{n} \circ T_{n},\ \theta] = 0}$. Consider ${\varepsilon > 0}$ and a subsequence of finite-dimensional subspaces ${(K_{n})_{\mathbb{N}^{''}} \prec (G_{A_{n} \circ T_{n}})_{\mathbb{N}^{''}}}$ from the product space ${W \times Y}$ such that ${\dim K_{n} \rightarrow \infty}$ -- we need to find a subsequence of subspaces ${(L_{n})_{\mathbb{N}^{'''}} \prec (K_{n})_{\mathbb{N}^{''}}}$ such that ${\dim L_{n} \rightarrow \infty}$ and ${\lambda_{\mathbb{N}^{''}}[L_{n}, \{\theta\}] \leq \varepsilon}$. Define subspaces ${Y_{n}}$ as images of natural projections from ${K_{n}}$ into $Y$ and ${W_{n}}$ as images of natural projections from ${K_{n}}$ into $W$ for each ${n \in \mathbb{N}^{''}}$. Then, since ${\dim K_{n} \rightarrow \infty}$, it must be that either ${\varlimsup \dim Y_{n} < \infty}$ or ${\varlimsup \dim Y_{n} = \infty}$. In the former case each $K_{n}$ becomes a graph of a finite rank operator for large enough ${n \in \mathbb{N}^{''}}$, so define each ${L_{n}}$ as a graph of such finite rank operator restricted to its kernel -- obviously kernel dimensions approach infinity, therefore ${\varlimsup \dim L_{n} \rightarrow \infty}$ and at the same time each $L_{n}$ is a graph of a null operator. Therefore ${\lambda_{\mathbb{N}^{''}}[L_{n}, \{\theta\}] = 0 < \varepsilon}$.

In the latter case of ${\dim Y_{n} \rightarrow \infty}$ consider subspaces ${P_{n} \subset G_{A_{n}}}$ defined as the graphs of restriction operators ${A_{n} \mid _{A_{n}^{-1}(Y_{n}) \cap W_{n}}}$ for each ${n \in \mathbb{N}^{''}}$ -- from assumption it is obvious that ${\dim P_{n} \rightarrow \infty}$. Therefore, since ${\mathcal{F}\mathcal{S}\mathcal{S}\lambda_{\mathbb{N}^{'}}[A_{n},\ \theta] = 0}$, there exists a subsequence of subspaces ${(Q_{n})_{\mathbb{N}^{'''}} \prec (P_{n})_{\mathbb{N}^{'''}}}$ such that ${\dim Q_{n} \rightarrow \infty}$ and that ${\lambda_{\mathbb{N}^{'''}}[Q_{n}, {\theta}] \leq \varepsilon \beta}$ where ${\beta = \frac{1}{\max(t, 1)}}$. Let ${R_{n}}$ be an image of a natural projection from ${Q_{n}}$ into $X$ -- obviously ${\dim R_{n} \rightarrow \infty}$. Now define ${S_{n} = W_{n} \cap T_{n}^{-1}(R_{n}) \subset W}$ and consider restriction operators ${V_{n} = A_{n} \circ T_{n} \mid _{S_{n}}}$ for every ${n \in \mathbb{N}^{'''}}$\ --\ obviously ${\dim G_{V_{n}} \rightarrow \infty}$ since ${\dim R_{n} \rightarrow \infty}$; also from our construction ${(G_{V_{n}})_{\mathbb{N}^{'''}} \prec (K_{n})_{\mathbb{N}^{'''}}}$. We shall prove that\linebreak ${\lambda_{\mathbb{N}^{'''}}[G_{V_{n}}, \{\theta\}] \leq \varepsilon}$\ --\ this will mean that ${\mathcal{F}\mathcal{S}\mathcal{S}\lambda_{\mathbb{N}^{'}}[A_{n} \circ T_{n},\ \theta] = 0}$. Let ${(w_{n}, A_{n} \circ T_{n} w_{n})_{\mathbb{N}^{'4}} \triangleleft (G_{V_{n}})_{\mathbb{N}^{'4}}}$ be a unit sequence of vectors. Let ${\alpha_{n} = \max(\left\|T_{n} w_{n}\right\|, \left\|A_{n} \circ T_{n} w_{n}\right\|)}$ for each ${n \in \mathbb{N}^{'4}}$ and consider a unit sequence ${(\alpha_{n}^{-1} T_{n} w_{n}, \alpha_{n}^{-1} A_{n} \circ T_{n} w_{n})_{\mathbb{N}^{'4}} \triangleleft (Q_{n})_{\mathbb{N}^{'4}}}$. Since ${\lambda_{\mathbb{N}^{'''}}[Q_{n}, \{\theta\}] \leq \varepsilon \beta}$ there exists a subsequence ${\mathbb{N}^{'5}}$ and a vector ${y \in Y}$ such that ${\varlimsup_{n \in \mathbb{N}^{'5}} \left\|\alpha_{n}^{-1} A_{n} \circ T_{n} w_{n} - y \right\| \leq \varepsilon \beta}$. By construction ${\alpha_{n} \leq \max(t, 1) = \beta^{-1}}$, hence there exists a subsequence ${(\alpha_{n})_{\mathbb{N}^{'6}} \rightarrow \alpha \leq \beta^{-1}}$. Therefore ${\varlimsup_{n \in \mathbb{N}^{'6}} \left\|A_{n} \circ T_{n} w_{n} - \alpha y \right\| \leq \varepsilon}$ which proves that ${\lambda_{\mathbb{N}^{'''}}[G_{V_{n}}, \{\theta\}] \leq \varepsilon}$.

In order to prove that ${\mathcal{F}\mathcal{S}\mathcal{S}\lambda_{\mathbb{N}^{'}}[S_{n} \circ A_{n},\ \theta] = 0}$ from the second equality let ${\omega > 0}$,\linebreak ${\varepsilon \in (0, \min(1, \omega s^{-1}))}$, then choose three positive numbers ${\gamma}$, ${\eta}$ and ${\nu}$ so that ${\gamma + \eta + \nu < \varepsilon2^{-1}}$ and consider a subsequence of subspaces ${(K_{n})_{\mathbb{N}^{''}}}$ from ${(G_{S_{n} \circ A_{n}})_{\mathbb{N}^{''}}}$ such that ${\dim K_{n} < \infty}$ for all ${n \in \mathbb{N''}}$ and ${\lim_{n \in \mathbb{N}^{''}} \dim K_{n} = \infty}$ -- our goal is to find a subsequence of subspaces $(L_{n})_{\mathbb{N}^{'''}} \prec (K_{n})_{\mathbb{N}^{'''}}$ such that ${\lim_{n \in \mathbb{N}^{'''}} \dim L_{n} = \infty}$ with the property ${\lambda_{\mathbb{N}^{'''}}[L_{n}, \{\theta\}] \leq \omega}$. Let ${X_{n}}$ be the image of a natural projection of ${K_{n}}$ onto $X$ -- obviously ${\dim X_{n} \rightarrow \infty}$. Let ${B_{n}}$ be a restriction of ${A_{n}}$ onto ${X_{n}}$. Obviously ${\dim G_{B_{n}} \rightarrow \infty}$. Therefore there exists a subsequence of subspaces ${(H_{n})_{\mathbb{N}^{'''}} \prec (G_{B_{n}})_{\mathbb{N}^{'''}}}$ such that ${\dim H_{n} \rightarrow \infty}$ and that ${\lambda_{\mathbb{N}^{'''}}[H_{n}, \{\theta\}] \leq \gamma}$. Obviously each $H_{n}$ is a graph of some operator ${\hat{A}_{n}}$ which is a restriction of the original operator $A_{n}$ so that ${\lambda_{\mathbb{N}^{'''}}[\hat{A}_{n}, \theta] \leq \gamma}$. As the main step in our proof we will establish existence of a subsequence of subspaces ${(X_{n})_{\mathbb{N}^{'4}} \prec (dom(\hat{A}_{n}))_{\mathbb{N}^{'4}}}$ such that ${\varlimsup \dim dom (\hat{A}_{n}) / X_{n} < \infty}$ and that ${\varlimsup \left\|\hat{A}_{n} \mid _{X_{n}} \right\| \leq \varepsilon}$. Suppose to the contrary that such a subsequence does not exist -- we shall then build a sequence of unit vectors ${(x_{n})_{\mathbb{N}^{'4}} \triangleleft (dom (\hat{A}_{n}))_{\mathbb{N}^{'4}}}$ and a sequence of subspaces ${(Y_{n})_{\mathbb{N}^{'4}} \subset Y}$ by induction like this:
\begin{itemize}
\item According to our assumption there exists a unit vector ${x_{n_{1}} \in dom(\hat{A}_{n_{1}})}$ such that\linebreak ${\left\|\hat{A}_{n_{1}}x_{n_{1}}\right\| \geq \varepsilon}$ for some ${n_{1} \in \mathbb{N}^{'''}}$. Using Hahn-Banach theorem find a unit functional ${f_{n_{1}} \in Y^{*}}$ such that ${f_{n_{1}}(\hat{A}_{n_{1}}x_{n_{1}}) = \left\|\hat{A}_{n_{1}}x_{n_{1}}\right\|}$. Then let ${y_{n_{1}} = \hat{A}_{n_{1}}x_{n_{1}}}$ and ${Y_{n_{1}} = Ker(f_{n_{1}})}$\ --\ obviously ${dist(\beta y_{n_{1}}, Y_{n_{1}}) \geq \beta \varepsilon}$ for any ${\beta > 0}$ and ${dist(\left\|y_{n_{1}}\right\|^{-1} y_{n_{1}}, Y_{n_{1}}) = 1}$.
\item Suppose that for some ${k > 1}$ we have already built a set of unit vector ${\{x_{n_{1}},...,x_{n_{k}}\} \subset X}$ and a set of subspaces ${Y \supset Y_{n_{1}} \supset ... \supset Y_{n_{k}}}$ such that ${x_{n_{i}} \in dom (\hat{A}_{n_{i}})}$, ${y_{n_{i}} = \hat{A}_{n_{i}}x_{n_{i}}}$,\linebreak ${Y_{n_{i}} = sp\{y_{n_{i+1}}\} \oplus Y_{n_{i+1}}}$ and ${dist(\beta y_{n_{i}}, Y_{n_{i}}) \geq \beta \varepsilon}$ for any ${\beta > 0}$ and ${dist(\left\|y_{n_{i}}\right\|^{-1} y_{n_{i}}, Y_{n_{i}}) = 1}$ for all ${i = 1,...,k-1}$. Obviously ${\dim dom(\hat{A}_{n_{k}} /  \hat{A}_{n_{k}}^{-1}(Y_{n_{k}})) \leq n_{k}}$ for each ${n \in \mathbb{N}^{'''}}$ -- therefore by assumption there exists some ${n_{k+1} \in \mathbb{N}^{'''}}$, ${n_{k+1} > n_{k}}$ and a unit vector ${x_{n_{k+1}} \in \hat{A}_{n_{k}}^{-1}(Y_{n_{k}})}$ such that ${\left\|\hat{A}_{n_{k+1}}x_{n_{k+1}}\right\| \geq \varepsilon}$. Denote ${y_{n_{k+1}} = \hat{A}_{n_{k+1}}x_{n_{k+1}}}$\ --\ by the Hahn-Banach theorem there exists a unit functional ${f_{n_{k+1}} \in Y_{n_{k}}^{*}}$ such that ${f_{n_{k+1}}y_{n_{k+1}} = \left\|y_{n_{k+1}}\right\|}$, so set\linebreak ${Y_{n_{k+1}} = Ker(f_{n_{k+1}})}$. Obviously ${Y_{n_{k}} = sp\{y_{n_{k+1}}\} \oplus Y_{n_{k+1}}}$ and ${dist(\beta y_{n_{k+1}}, Y_{n_{k+1}}) \geq \beta \varepsilon}$ for any ${\beta > 0}$ and ${dist(\left\|y_{n_{k+1}}\right\|^{-1} y_{n_{k+1}}, Y_{n_{k+1}}) = 1}$.
\end{itemize}
Now define a sequence of just built numbers ${\mathbb{N}^{'4} = \{n_{1},...,n_{k},...\}}$. Then for each ${n \in \mathbb{N}^{'4}}$ define ${\beta_{n} = \max(1, \left\|y_{n}\right\|)}$ and consider a unit sequence
\[
((\beta_{n}^{-1}x_{n}, \beta_{n}^{-1}\hat{A}_{n}x_{n}))_{\mathbb{N}^{'4}}\ =\ ((\beta_{n}^{-1}x_{n}, \beta_{n}^{-1}y_{n}))_{\mathbb{N}^{'4}}\ \subset\ X \times Y.
\]
Since ${\lambda_{\mathbb{N}^{'''}}[\hat{A}_{n}, \theta] \leq \gamma}$ there exists a subsequence ${\mathbb{N}^{'5}}$ and a vector ${y \in Y}$ such that
\[
\varlimsup_{\mathbb{N}^{'5}} \left\|\beta_{n}^{-1}y_{n} - y \right\|\ \leq\ \gamma + \eta.
\]
Applying the triangle inequality for large enough ${n > m}$ from ${\mathbb{N}^{'5}}$ and taking into account the choice of numbers ${\gamma}$, ${\eta}$ and ${\nu}$ estimate
\begin{multline}
\notag
\begin{aligned}
\left\|\beta_{m}^{-1}y_{m} - \beta_{n}^{-1}y_{n}\right\|\ &=\ \left\|\beta_{m}^{-1}y_{m} - y - \beta_{n}^{-1}y_{n} + y\right\|\ \leq\ \left\|\beta_{m}^{-1}y_{m} - y\right\| + \left\|\beta_{n}^{-1}y_{n} - y\right\|\\
                                                          &\leq\ \gamma + \eta + \nu + \gamma + \eta + \nu\ =\ 2 \times(\gamma + \eta + \nu)\ <\ \varepsilon.                                                     
\end{aligned}
\end{multline}
At the same time by construction
\[
\left\|\beta_{m}^{-1}y_{m} - \beta_{n}^{-1}y_{n}\right\|\ \geq\ \beta_{m}^{-1} \times \varepsilon.
\]
Thus from the above two inequalities obtain
\[
\beta_{m}^{-1} \times \varepsilon\ <\ \varepsilon.
\]
If ${\left\|y_{m}\right\| \leq 1}$ then ${\beta_{m} = 1}$, therefore ${\varepsilon < \varepsilon}$ which is a contradiction. So it must be that ${\left\|y_{m}\right\| > 1}$ -- then it follows that ${\beta_{m} = \left\|y_{m}\right\|}$ and therefore from the construction
\[
\varepsilon\ >\ \left\|\beta_{m}^{-1}y_{m} - \beta_{n}^{-1}y_{n}\right\|\ =\ \left\|\left\|y_{m}\right\|^{-1}y_{m} - \beta_{n}^{-1}y_{n}\right\|\ \geq\ 1
\]
which contradicts our choise of ${\varepsilon < 1}$. Thus we have proved that our assumption is incorrect, and, therefore there exists a subsequence of subspaces ${(X_{n})_{\mathbb{N}^{'''}} \prec dom(\hat{A}_{n})_{\mathbb{N}^{'''}}}$ such that ${\varlimsup \dim dom (\hat{A}_{n}) / X_{n} < \infty}$ and that ${\varlimsup \left\|\hat{A}_{n} \mid _{X_{n}} \right\| \leq \varepsilon}$. Therefore, ${\varlimsup \left\|S_{n} \circ \hat{A}_{n} \mid _{X_{n}} \right\| \leq \omega}$ by the choice of ${\varepsilon < \omega s^{-1}}$ and due to the condition ${\varlimsup \left\|S_{n}\right\| = s < \infty}$. Thus, if we choose ${L_{n} = G_{S_{n} \circ \hat{A}_{n} \mid _{X_{n}}}}$ it is clear that $(L_{n})_{\mathbb{N}^{'''}} \prec (K_{n})_{\mathbb{N}^{'''}}$, ${\lim_{n \in \mathbb{N}^{'''}} \dim L_{n} = \infty}$ and ${\lambda_{\mathbb{N}^{'''}}[L_{n}, \{\theta\}] \leq \omega}$ which concludes the proof of the second equality from proposition $1$.

We omit the proof of the proposition $2$ as it is practically the same as the just concluded proof of the proposition $1$.
\end{proof}

\section[Semi-Fredholm Stability]{Semi-Fredholm Stability}\label{S:sfsulssa}
In this section we prove stability of lower semi-Fredholm pairs and of operators under lower strictly singular $\lambda-$adjustment. Let us start with some preliminaries and a few lemmas that will be used in that proof.
\subsection[Preliminary Lemmas]{Preliminary Lemmas}\label{SS:pl}
The following concept of gap between two subspaces of a Banach space had been introduced  by M. G. Krein, M. A. Krasnosel'ski\~{\i}, D. P. Mil'man in \cite{krein_krasnoselskii_milman} (see also M. I. Ostrovskii \cite{ostrovskii} and T. Kato \cite{kato}) -- it can be seen as a measure of an 'angle' between two subspaces:
\begin{definition}[Gap between Two Subspaces of a Banach Space]\label{bgtsbs}
Recall that if $x$ is a vector from a Banach space $X$, and $P$ is a closed subspace of $X$, then \emph{distance from $x$ to $P$} is defined as ${dist(x, P) := \inf \{ \left\| x - y \right\| \mid y \in P \}}$; if $M$, $P$ are two closed subspaces of a Banach space $X$ then the $gap$\ $\delta(M, P)$\ $between$\ $M$\ $and$\ $P$ is a real non-negative number defined as
\[ 
\delta(M, P)\ :=\ \sup \{ dist(x, P) \mid x \in M ,\ \left\| x \right\| = 1 \}.
\]
\end{definition}
The concept of uniform $\lambda-$adjustment is weaker than the concept of gap distance (see \cite{burshteyn})
\[
	\lambda_{\mathbb{N}^{'}}[M_{n}, P_{n}]\ \le\ \varlimsup_{n \in \mathbb{N}^{'}}  \delta(M_{n}, P_{n}).
\]
The following theorem due to I.C. Gohberg, M.G. Krein from \cite{gohberg_krein} had become a fundamental tool in the research of (semi-)Fredholm stability:
\begin{theorem}[The Small Gap Theorem]\label{tsgp}
Let $M$ and $N$ be two subspaces from a Banach space such that ${\delta(M, N) < 1}$. If ${\dim M < \infty}$ then ${\dim M \leq \dim N}$.
\end{theorem}
This theorem is used in order to prove the following technical lemma:
\begin{lemma}\label{253}
Let ${(M_{n})_{\mathbb{N}^{'}}}$ and ${(N_{n})_{\mathbb{N}^{'}}}$ be two sequences of subspaces from a Banach space. Suppose that ${\lim_{n \in \mathbb{N}^{'}} \delta(M_{n}, N_{n}) = 0}$ and ${\lim_{n \in \mathbb{N}^{'}} \dim M_{n} = \infty}$. Then there exist two subsequences of subspaces ${(H_{n})_{\mathbb{N}^{'}_{1}} \prec (M_{n})_{\mathbb{N}^{'}_{1}}}$ and ${(G_{n})_{\mathbb{N}^{'}_{1}} \prec (N_{n})_{\mathbb{N}^{'}_{1}}}$ such that ${\dim H_{n} = \dim G_{n}}$ for all ${n \in \mathbb{N}^{'}_{1}}$, ${\lim_{n \in \mathbb{N}^{'}_{1}} \dim H_{n} = \infty}$ and ${\lim_{n \in \mathbb{N}^{'}_{1}} \delta(H_{n}, G_{n}) = \lim_{n \in \mathbb{N}^{'}_{1}} \delta(G_{n}, H_{n}) = 0}$.
\begin{proof}
Since ${\lim_{n \in \mathbb{N}^{'}} \dim M_{n} = \infty}$ we can choose a subsequence ${\mathbb{N}^{''} = \{n_{1}  < ... < n_{k} < ...\} \subset \mathbb{N}^{'}}$ such that ${\dim M_{n_{k}} \geq k}$ for each ${k \in \mathbb{N}}$. Since ${\lim_{n \in \mathbb{N}^{'}} \delta(M_{n}, N_{n}) = 0}$ we can choose a subsequence ${\mathbb{N}^{'''} = \{m_{1} < m_{2} < ... < m_{k} < ...\} \subset \mathbb{N}^{''}}$ such that ${\dim M_{m_{k}} \geq k}$ and
\[
\delta(M_{m_{k}}, N_{m_{k}})\ <\ (k+1)^{-1} \times 2^{-(k+1)}
\]
for each ${k \in \mathbb{N}}$. 

Since ${\dim M_{m_{k}} \geq k}$ we can choose a unit vector ${x_{m_{k}}^{1} \in M_{m_{k}}}$ and since\linebreak ${\delta(M_{m_{k}}, N_{m_{k}}) < (k+1)^{-1} \times 2^{-(k+1)}}$ we can choose a vector ${y_{m_{k}}^{1} \in N_{m_{k}}}$ such that
\[
\left\|x_{m_{k}}^{1} - y_{m_{k}}^{1}\right\|\ <\ (k+1)^{-1} \times 2^{-(k+1)}
\]
for each ${k \geq 2}$. By the Hahn-Banach theorem there exists a unit functional ${f^{1}_{k} \in M_{m_{k}}^{*}}$ such that ${\| f^{1}_{k}x_{m_{k}}^{1} \| = 1}$ with the kernel space ${K_{m_{k}}^{1} \subset M_{m_{k}}}$. Obviously ${M_{m_{k}} = sp\{x_{m_{k}}^{1}\} \oplus K_{m_{k}}^{1}}$, so we can define a projection ${P_{k}^{1} : M_{m_{k}} \rightarrow sp\{x_{m_{k}}^{1}\}}$ such that ${\left\|P_{k}^{1}\right\| = 1}$, ${P_{k}^{1}x = (f^{1}_{k}x) x}$ and ${Ker(P_{k}^{1}) = K_{m_{k}}^{1}}$ for each ${k \geq 2}$. 

Obviously ${\dim K_{m_{k}}^{1} = k - 1}$ for each ${k \geq 2}$. Therefore, for each ${k \geq 3}$ we can apply the same reasoning and choose a unit vector ${x^{2}_{m_{k}} \in K_{m_{k}}^{1}}$ and a vector ${y_{m_{k}}^{2} \in N_{m_{k}}}$ such that\linebreak ${\left\|x_{m_{k}}^{2} - y_{m_{k}}^{2}\right\| < (k+1)^{-1} \times 2^{-(k+1)}}$. Also, applying the Hahn-Banach theorem we can find a unit functional ${f^{2}_{k} \in K^{1*}_{m_{k}}}$ such that ${\| f^{2}_{k}x_{m_{k}}^{2} \| = 1}$ with the kernel space ${K_{m_{k}}^{2} \subset M_{m_{k}}}$. Obviously ${K^{1}_{m_{k}} = sp\{x_{m_{k}}^{2}\} \oplus K_{m_{k}}^{2}}$, so we can define a projection ${P_{k}^{2} : K_{m_{k}}^{1} \rightarrow sp\{x_{m_{k}}^{2}\}}$ such that ${\left\|P_{k}^{2}\right\| = 1}$, ${P_{k}^{2}x = (f^{2}_{k}x) x}$ and ${Ker(P_{k}^{2}) = K_{m_{k}}^{2}}$ for each ${k \geq 3}$. 

Naturally following this process we end up with the following objects for each number ${k \geq 2}$: a set of linearly independent unit vectors ${\{x^{1}_{m_{k}}, ..., x^{k}_{m_{k}}\} \in M_{m_{k}}}$ and a set of vectors\linebreak ${\{y^{1}_{m_{k}}, ..., y^{k}_{m_{k}}\} \in N_{m_{k}}}$ such that
\[
\left\|x_{m_{k}}^{i} - y_{m_{k}}^{i}\right\|\ <\ (k+1)^{-1} \times 2^{-(k+1)}
\]
for each ${i = 1, ..., k}$, as well as the set of unit projections
\[
P^{i}_{m_{k}} : sp\{x^{i}_{m_{k}}, ..., x^{k}_{m_{k}}\} \rightarrow sp\{x^{i}_{m_{k}}\}
\]
such that ${ker(P^{i}_{m_{k}}) = sp\{x^{i+1}_{m_{k}}, ..., x^{k}_{m_{k}}\}}$ for each ${i = 1, ..., k-1}$. Let ${I^{i}_{m_{k}}}$ be a an identity operator for the space ${sp\{x^{i+1}_{m_{k}}, ..., x^{k}_{m_{k}}\}}$ for each ${i = 1, ..., k-1}$ and consider any vector\linebreak ${x = \sum_{i=1}^{k}\alpha_{i}x_{m_{k}}^{i} \in sp\{x^{1}_{m_{k}}, ..., x^{k}_{m_{k}}\}}$. Estimate each ${\alpha_{i}}$:
\begin{multline}
\notag
\begin{aligned}
| \alpha_{i} |\ &=\ \left\|P^{i}_{m_{k}} \circ (I^{i-1}_{m_{k}} - P^{i-1}_{m_{k}}) \circ ... \circ (I^{1}_{m_{k}} - P^{1}_{m_{k}}) x \right\|\\
                &\leq\ \left\|P^{i}_{m_{k}}\right\| \times \left\|\prod_{j=1}^{i-1} (I^{i-j}_{m_{k}} - P^{i-j}_{m_{k}})\right\| \times \left\| x \right\|\ \leq\ 1 \times (\prod_{j=1}^{i-1} \left\| I^{i-j}_{m_{k}} - P^{i-j}_{m_{k}} \right\|) \times \left\| x \right\|\\
                &\leq\ (\prod_{j=1}^{i-1} (\left\| I^{i-j}_{m_{k}} \right\| + \left\| P^{i-j}_{m_{k}} \right\| )) \times \left\| x \right\|\ \leq\ (\prod_{j=1}^{i-1} 2 ) \times \left\| x \right\|\ <\  2^{i} \times \left\| x \right\|;       
\end{aligned}
\end{multline}
therefore
\begin{multline}\label{E:1}
\begin{aligned}
&\sum_{i=1}^{k} |\alpha_{i}|\ <\ (\sum_{i=1}^{k} 2 ^{i} ) \times \left\| x \right\|\ <\ 2^{k+1} \times \left\| x \right\|,\\
&\frac{\sum_{i=1}^{k} |\alpha_{i}|}{2^{k+1}}\ <\ \left\| x \right\|.
\end{aligned}
\end{multline}
Now consider pairs of subspaces ${H_{m_{k}} = sp\{x^{1}_{m_{k}}, ..., x^{k}_{m_{k}}\} \subset M_{m_{k}}}$ and ${G_{m_{k}} = sp\{y^{1}_{m_{k}}, ..., y^{k}_{m_{k}}\} \subset N_{m_{k}}}$ and estimate the gap ${\delta(G_{m_{k}}, H_{m_{k}})}$ between $G_{m_{k}}$ and $H_{m_{k}}$. Take a unit vector ${y \in G_{m_{k}}}$ -- let us calculate $y$'s distance from ${H_{m_{k}}}$. If ${y = \sum_{i=1}^{k} \alpha_{i} y^{i}_{m_{k}}}$ then take a vector from ${x \in H_{m_{k}}}$ such that ${x = \sum_{i=1}^{k} \alpha_{i} x^{i}_{m_{k}}}$ and evaluate its norm:
\begin{multline}\label{E:2}
\begin{aligned}
\left\|x\right\|\ &=\ \left\|x + y - y\right\|\ =\ \left\|\sum_{i=1}^{k} \alpha_{i} x^{i}_{m_{k}} + \sum_{i=1}^{k} \alpha_{i} y^{i}_{m_{k}} - \sum_{i=1}^{k} \alpha_{i} y^{i}_{m_{k}}\right\|\\
                  &\leq\ \left\|\sum_{i=1}^{k} \alpha_{i} y^{i}_{m_{k}}\right\|\ +\ \left\|\sum_{i=1}^{k} \alpha_{i} x^{i}_{m_{k}} - \sum_{i=1}^{k} \alpha_{i} y^{i}_{m_{k}}\right\|\ =\ \left\|y\right\|\ +\ \left\|\sum_{i=1}^{k} \alpha_{i}(x^{i}_{m_{k}} - y^{i}_{m_{k}})\right\|\\
                  &\leq\ 1\ +\ \sum_{i=1}^{k} |\alpha_{i}| \left\|x^{i}_{m_{k}} - y^{i}_{m_{k}}\right\|\ \leq\ 1\ +\ (\sum_{i=1}^{k} |\alpha_{i}|) \times (k+1)^{-1}2^{-(k+1)}.
\end{aligned}
\end{multline}
Combining estimates for $x$'s norm from \eqref{E:1} and \eqref{E:2} calculate further:
\begin{multline}\label{E:3}
\begin{aligned}
&\frac{\sum_{i=1}^{k} |\alpha_{i}|}{2^{k+1}}\ <\ 1\ +\ (\sum_{i=1}^{k} |\alpha_{i}|) \times (k+1)^{-1}2^{k+1},\\
&(\sum_{i=1}^{k} |\alpha_{i}|) \times (2^{-(k+1)} - (k+1)^{-1}2^{-(k+1)})\ <\ 1,\\
&\sum_{i=1}^{k} |\alpha_{i}|\ <\ \frac{2^{k+1}(k+1)}{k}.
\end{aligned}
\end{multline}
Finally we can estimate using \eqref{E:3} 
\begin{multline}
\notag
\begin{aligned}
dist(y, H_{m_{k}})\ &\leq\ \left\|y - x\right\|\ =\ \left\|\sum_{i=1}^{k} \alpha_{i} y^{i}_{m_{k}} - \sum_{i=1}^{k} \alpha_{i} x^{i}_{m_{k}}\right\|\ \leq\ \sum_{i=1}^{k} |\alpha_{i}|\left\|y_{i} - x_{i}\right\|\\
                   &\leq\ (\sum_{i=1}^{k} |\alpha_{i}|) \times (k+1)^{-1}2^{-(k+1)}\ <\ \frac{2^{k+1}(k+1)}{k} \times (k+1)^{-1}2^{-(k+1)}\ =\ \frac{1}{k},\\
\end{aligned}
\end{multline}
therefore
\begin{multline}\label{E:4}
\begin{aligned}
\delta(G_{m_{k}}, H_{m_{k}})\ <\ \frac{1}{k}.
\end{aligned}
\end{multline}
Now estimate the opposite gap ${\delta(H_{m_{k}}, G_{m_{k}})}$ by taking a unit vector ${x = \sum_{i=1}^{k} \alpha_{i}x^{i}_{m_{k}} \in H_{m_{k}}}$ and evaluating its difference with ${y = \sum_{i=1}^{k} \alpha_{i}y^{i}_{m_{k}} \in G_{m_{k}}}$ using \eqref{E:1}:
\begin{multline}
\notag
\begin{aligned}
dist(x, G_{m_{k}})\ &\leq\ \left\|x - y\right\|\ =\ \left\|\sum_{i=1}^{k} \alpha_{i} x^{i}_{m_{k}} - \sum_{i=1}^{k} \alpha_{i} y^{i}_{m_{k}}\right\|\ \leq\ \sum_{i=1}^{k} |\alpha_{i}|\left\|x_{i} - y_{i}\right\|\\
                   &\leq\ (\sum_{i=1}^{k} |\alpha_{i}|) \times (k+1)^{-1}2^{-(k+1)}\ <\ 2^{k+1} \times (k+1)^{-1}2^{-(k+1)}\ =\ (k+1)^{-1},
\end{aligned}
\end{multline}
therefore
\begin{multline}\label{E:5}
\begin{aligned}
\delta(H_{m_{k}}, G_{m_{k}})\ <\ \frac{1}{k+1}.
\end{aligned}
\end{multline}
Now denote ${\mathbb{N}^{'}_{1} := \mathbb{N}^{'''} = \{m_{1} < m_{2} < ... < m_{k} < ...\}}$. Then from \eqref{E:4} and \eqref{E:5} it follows that
\[
\lim_{n \in \mathbb{N}^{'}_{1}} \delta(H_{n}, G_{n})\ =\ \lim_{n \in \mathbb{N}^{'}_{1}} \delta(G_{n}, H_{n})\ =\ 0.
\]
From the above formula and from the fact that ${\dim H_{m_{k}} = k}$, as well as from the small gap theorem \ref{tsgp} applied to pairs of subspaces ${(H_{m_{k}}, G_{m_{k}})}$ it follows that
\[
\dim H_{m_{k}}\ =\ \dim G_{m_{k}}\ =\ k
\]
for ${k \in \mathbb{N}}$\ --\ this concludes the proof of the lemma. 
\end{proof}
\end{lemma}
The next lemma investigates the structure of a pair of subspaces which sum is not closed:
\begin{lemma}[Having a Not Closed Sum is an Inherent Property]\label{hncsip}
Let ${(M, N)}$ be a pair of closed subspaces from a Banach space such that they have no common non-null elements, i.e. ${M \cap N = \{\theta\}}$, and such that the sum of $M$ and $N$ is not closed, i.e. ${M+N \neq \overline{M+N}}$. Then there exists a pair ${(H, G)}$ of infinite-dimensional closed subspaces ${H \subset M}$, ${G \subset N}$ such that for any infinite-dimensional subspace ${K \subset H}$ there exists an infinite-dimensional subspace ${L \subset G}$ such that the sum of $K$ and $L$ is not closed, i.e. ${K+L \neq \overline{K+L}}$. Moreover for any ${\varepsilon > 0}$ one can choose subspaces $H$ and $G$ so that ${\max(\delta(H,G), \delta(G,H)) < \varepsilon}$; also subspace $L$ can be chosen so that ${\max(\delta(K,L), \delta(L,K)) < \varepsilon}$.
\end{lemma}
\begin{proof}
First we build by induction a system of the following objects: a sequence of decreasing closed subspaces ${M = M_{1} \supset M_{2} \supset ... \supset M_{n} \supset ...}$\ and a sequence of unit vectors ${(x_{n})_{\mathbb{N}} \triangleleft (M_{n})_{\mathbb{N}}}$ such that ${M_{i} = sp\{x_{i}\} \oplus M_{i+1}}$ and ${dist(x_{i}, M_{i+1}) = 1}$ for each ${i \in \mathbb{N}}$, as well as a sequence of vectors ${(y_{i})_{\mathbb{N}} \subset N}$ such that ${\left\|x_{i} - y_{i}\right\| < 2^{-2i}}$ for each ${i \in \mathbb{N}}$. For that consider a continuous linear map ${\Psi_{1} : M \times N \rightarrow M + N}$ defined as ${\Psi_{1}:(x,y) \mapsto x-y}$. Since ${M \cap N = \{\theta\}}$ and ${M+N \neq \overline{M+N}}$ we conclude that $\Psi_{1}$ is an injection and that the image of $\Psi_{1}$ is not closed; therefore, by the open mapping theorem for every ${\varepsilon > 0}$ there exists a unit pair ${(x,y) \in M \times N}$ such that ${\left\|x\right\| = 1}$ and ${\left\|x-y\right\| < \varepsilon}$. Thus, for the first induction step choose a unit pair ${(x_{1},y_{1}) \in M \times N}$ such that ${\left\|x_{1}\right\| = 1}$ and ${\left\|x_{1}-y_{1}\right\| < 2^{-2}}$; then denote ${M_{1} = M}$ and by the Hahn-Banach theorem choose a closed subspace ${M_{2}}$ as the kernel of a continuous unit functional ${f_{1} \in M_{1}^{*}}$ such that ${f_{1}x_{1} = 1}$. Now suppose that the first $n+1$ closed spaces ${M_{i}}$ and the first $n$ vectors ${x_{i}}$, ${y_{i}}$ have been already built. Since ${\dim M / M_{n} = n}$ and ${M+N \neq \overline{M+N}}$ we shall conclude that ${M_{n}+N \neq \overline{M_{n}+N}}$. Then, again as in the first step, with the help of the open mapping theorem applied to the natural linear map ${\Psi_{n} : M_{n} \times N \rightarrow M_{n} + N}$ defined as ${\Psi_{n}:(x,y) \mapsto x-y}$ we might choose a unit vector ${x_{n+1} \in M_{n}}$ and a vector ${y_{n+1} \in N}$ such that ${\left\|x_{n+1}-y_{n+1}\right\| < 2^{-2(n+1)}}$. Then by the Hahn-Banach theorem choose a closed subspace ${M_{n+2}}$ as the kernel of a continuous unit functional ${f_{n+1} \in M_{n+1}^{*}}$ such that ${f_{n+1}x_{n+1} = 1}$ which concludes the induction step.

After that, define the sequence of subspaces ${M \supset S_{1} \supset S_{2} \supset ... \supset S_{n} \supset ...}$\ as\linebreak ${S_{n} = sp\{x_{n}, x_{n+1}, ...\} \subset M_{n}}$; then define identity injection maps ${I_{n} : S_{n} \rightarrow X}$ and linear maps ${A_{n} : S_{n} \rightarrow N}$ by setting ${A_{n}: \sum \alpha_{j}x_{j} \mapsto \sum \alpha_{j}y_{j}}$. Our goal is to prove that each operator ${A_{n}}$ is a continuous isomorphism and that ${\left\|I_{n} - A_{n}\right\| < 2^{-2n}}$. For that consider the sequence of projections ${P_{n} : M_{n} \rightarrow sp\{x_{n}\}}$ with the kernels ${Ker(P_{n}) = M_{n+1}}$ for each ${n \in \mathbb{N}}$ -- from the above construction it is clear that ${\left\|P_{n}\right\| = 1}$ for each projection ${P_{n}}$. Then for each vector ${x = \sum_{i=0}^{k} \alpha_{n+i}x_{n+i} \in S_{n}}$ estimate its components ${\alpha_{n+i}}$ for each ${i = 0,...,k}$ like this:
\begin{multline}
\notag
\begin{aligned}
|\alpha_{n+i}|\ &=\ \left\|P_{n} \circ (I_{n} - P_{n+1}) \circ ... \circ (I_{n+i-1} - P_{n+i}) x\right\|\\
                &\leq\ \left\|P_{n}\right\| \times \left\| \prod_{j=0}^{i-1}(I_{n+j} - P_{n+j+1})\right\| \times \left\|x\right\|\ =\ 1 \times \left\|\prod_{j=0}^{i-1}(I_{n+j} - P_{n+j+1})\right\| \times \left\|x\right\|\\
                &\leq\ (\prod_{j=0}^{i-1}\left\|I_{n+j} - P_{n+j+1}\right\|) \times \left\|x\right\|\ \leq\ (\prod_{j=0}^{i-1}(\left\|I_{n+j}\right\| + \left\|P_{n+j+1}\right\|)) \times \left\|x\right\|\\
                &=\ (\prod_{j=0}^{i-1}2) \times \left\|x\right\|\ =\ 2^{i-1} \times \left\|x\right\|.
\end{aligned}
\end{multline}
Based on that get the desired estimate ${\left\|I_{n} - A_{n}\right\| < 2^{-2n}}$: 
\begin{multline}
\notag
\begin{aligned}
\left\|x - A_{n}x\right\|\ &=\ \left\|\sum_{i=0}^{k} \alpha_{n+i}x_{n+i} - \sum_{i=0}^{k} \alpha_{n+i}y_{n+i}\right\|\ \leq\ \sum_{i=0}^{k} |\alpha_{n+i}| \times(\left\|x_{n+i} - y_{n+i}\right\|)\\
                           &\leq\ \sum_{i=0}^{k} |\alpha_{n+i}| \times 2^{-2(n+i)}\ \leq\ \sum_{i=0}^{k} 2^{i-1} \times \left\|x\right\| \times 2^{-2(n+i)}\ =\ (\sum_{i=0}^{k} 2^{-2n -i -1}) \times \left\| x \right\|\\
                           &=\ 2^{-2n-1} \times (\sum_{i=0}^{k}2^{-i}) \times \left\|x\right\|\ <\ 2^{-2n-1} \times 2 \times \left\|x\right\|\ =\ 2^{-2n} \times \left\|x\right\|.
\end{aligned}
\end{multline}
Since all ${I_{n}}$ and ${A_{n}}$ are bounded linear operators each defined on ${S_{n}}$ and since $N$ is a closed subspace, we may extend each of them by continuity onto ${\overline{S_{n}}}$ (the closure of ${S_{n}}$): ${I_{n}}$ to identity operators ${\hat{I}_{n} : \overline{S_{n}} \rightarrow \overline{S_{n}}}$ and ${A_{n}}$ to ${\hat{A}_{n} : \overline{S_{n}} \rightarrow N}$ for each ${n \in \mathbb{N}}$. Obviously the just proved estimate remains true for the extensions, that is ${\left\|\hat{I}_{n} - \hat{A}_{n}\right\| < 2^{-2n}}$. Also note that it follows from the properties of the above inductive construction that ${\overline{S_{1}} = sp\{x_{1},...,x_{n}\} \oplus \overline{S_{n}}}$ which implies that ${\dim \overline{S_{1}} / \overline{S_{n}} = n}$.

Now for every ${n \in \mathbb{N}}$ denote ${H_{n} = \overline{S_{n}} \subset M}$ and ${G_{n} = \hat{A}_{n}(H_{n}) \subset N}$. We shall prove that if ${2^{-n} < \varepsilon}$ then the pair of subspaces ${(H_{n}, G_{n})}$ satisfies the conditions of the lemma. First note that $H_{n}$ is an infinite-dimensional closed subspace by the previous construction. At the same time ${\left\|\hat{I}_{n} - \hat{A}_{n}\right\| < 2^{-2n}}$ while ${\hat{I}_{n}}$ is an isomorphism of norm $1$, therefore due to the stability of the isomorphism perturbed by a small norm operator we conclude that ${\hat{A}_{n}}$ is also an isomorphism and
\begin{multline}\label{E:1111}
\begin{aligned}
\delta(H_{n}, G_{n})\ &=\ \delta(\hat{I}_{n}(H_{n}), \hat{A}_{n}(H_{n}))\ \leq\ \left\|\hat{I}_{n} - \hat{A}_{n}\right\|\ \leq\ 2^{-2n}\ <\ 2^{-n}\ <\ \varepsilon,\\
\delta(G_{n}, H_{n})\ &=\ \delta(\hat{A}_{n}^{-1}(H_{n}), \hat{I}_{n}(H_{n}))\ \leq\ \left\|\hat{I}_{n} - \hat{A}_{n}^{-1}\right\|\ \leq\ \frac{\left\|\hat{I}_{n} - \hat{A}_{n}\right\|}{\left\|\hat{A}_{n}\right\|}\ \leq\ \frac{2^{-2n}}{1 - 2^{-2n}}\ <\ 2^{-n}\ <\ \varepsilon.
\end{aligned}
\end{multline}
Therefore ${G_{n} = \hat{A}_{n}(H_{n})}$ is a closed subspace and ${\max(\delta(H_{n}, G_{n}), \delta(G_{n}, H_{n})) < \varepsilon}$.

Now consider any infinite-dimensional closed subspace ${K \subset H_{n}}$ -- our goal is to find an infinite-dimensional subspace ${L \subset G_{n}}$ such that the sum of $K$ and $L$ is not closed and\linebreak ${\max(\delta(K,L), \delta(L,K)) < \varepsilon}$. For that choose $L$ to be the image of $K$, that is set ${L = \hat{A}_{n}(K)}$ -- since ${\hat{A}_{n}}$ is an isomorphism, $L$ is a closed infinite-dimensional subspace from $G_{n} \subset N$ and the gap estimate ${\max(\delta(K,L), \delta(L,K)) < \varepsilon}$ is true as in \eqref{E:1111}. 

Now let us prove that ${K+L \neq \overline{K+L}}$. Since ${K \cap L = \{\theta\}}$ we may consider, as in the above inductive construction, a natural linear injection ${\Psi : K \times L \rightarrow K + L}$ defined as ${\Psi:(x,y) \mapsto x-y}$. According to open mapping theorem it is sufficient to prove that for every ${\gamma> 0}$ there exist a unit pair vector ${(x,y) \in K \times L}$ such that ${\left\|x - y\right\| < \gamma}$. Let ${m > n}$, ${m \in \mathbb{N}}$ such that ${2^{-2m} < \gamma}$. Consider the space ${K_{m} = \overline{{S}_{m}} \cap K}$. Since ${\dim K = \infty}$, ${K \subset \overline{S_{n}}}$ and by construction\linebreak ${\dim \overline{S_{n}} / \overline{S_{m}} = m - n < \infty}$, it is clear that ${\dim K_{m} = \dim \overline{S_{m}} \cap K = \infty}$. Therefore, we may choose any unit vector ${(x, \hat{A}_{m}x) \in K_{m} \times \hat{A}_{m}(K_{m}) \subset K \times L}$ and let ${y = \hat{A}_{m}x \in L}$. By the previous norm estimate
\[
\left\|x - y\right\|\ =\ \left\|x - \hat{A}_{m}x\right\|\ \leq\ \left\|\hat{I}_{m} - \hat{A}_{m}\right\| \times \left\|x\right\|\ <\ 2^{-2m} \times 1\ \leq\ \gamma
\]
which concludes the proof of the entire lemma.
\end{proof}
\subsection[The Main Stability Theorem]{The Main Stability Theorem}\label{tmst}
At this point we have gathered almost all the technical facts needed for the proof of semi--Fredholm stability. As the final preliminary step we recall some technical considerations from \cite{burshteyn}. Let $M$, $N$ and $S$ be three closed subspaces from a Banach space $X$ such that[3] ${M + N = \overline{M + N}}$, ${M+N = (M \cap N) \oplus S}$. Consider a natural mapping
\[
\Phi\ :\ \Pi\ =\ M \cap S\ \times\ M \cap N\ \times\ N \cap S\ \rightarrow\ M + N\\
\]
\[
\Phi\ :\ (u, t, v)\ \rightarrow\ u + t + v.
\]
Define a complete norm on ${\Pi}$
\[
\left\|(u, t, v)\right\|\ =\ \max \{\left\|u\right\|,\left\|t\right\|,\left\|v\right\|\}.
\]
Then, since
\[
\overline{M + N}\ =\ M + N\ =\ (M \cap N) \oplus S\ =\ M \cap S\ \oplus\ M \cap N\ \oplus\ N \cap S,
\]
it is clear that $\Phi$ is a continuous bijection from the Banach space ${\Pi}$ onto the Banach space ${M + N}$. Therefore, according to the open mapping theorem operator ${\Phi^{-1}}$ is a continuous operator. Let us denote ${\varphi_{S}(M, N)\ :=\ \left\|\Phi^{-1}\right\|}$.

Also recall that a pair of closed subspaces $(M,N)$ from a Banach space $X$ is called \emph{lower semi--Fredholm} if the sum of $M$ and $N$ is closed, i.e. ${M+N = \overline{M+N}}$, and if its \emph{lower defect number} ${\alpha(M,N) := \dim(M \cap N)}$ is finite. Note that for a lower semi--Fredholm pair there always exist many closed subspaces $S$ such that ${M+N = (M \cap N) \oplus S}$. Therefore, denote
\[
\varphi(M, N)\ :=\ \inf \{ \varphi_{S}(M, N)\ \mid\ M+N = (M \cap N) \oplus S\}.
\]
What follows is the final technical lemma which proof can be found in Theorem $2.5.1$ from \cite{burshteyn}:
\begin{lemma}\label{ftl}
Let $M$, $N$ be two closed subspaces from a Banach space $X$ such that\linebreak ${M + N = \overline{M + N}}$. Let ${(M_{n})_{\mathbb{N}^{'}}}$, ${(N_{n})_{\mathbb{N}^{'}}}$ are two sequences of closed subspaces from $X$, and set
\[
\lambda_{M}\ =\ \lambda_{\mathbb{N}^{'}}[M_{n}, M],\ \ \ \lambda_{N}\ =\ \lambda_{\mathbb{N}^{'}}[N_{n}, N];
\]
then the following propositions are true
\begin{enumerate}
  \item Suppose that ${M+N = (M \cap N) \oplus S}$ and define a real number
\[
\omega_{S}(M)\ =\ 2 \times (\lambda_{M} + \varphi_{S}(M, N) \times (\lambda_{M} + \lambda_{N})).
\]  
If ${(H_{n})_{\mathbb{N^{''}}}}$ is a sequence of closed subspaces, ${H_{n} \subset M_{n}}$ for all ${n \in \mathbb{N}^{''}}$ and ${\delta(H_{n}, N_{n}) \rightarrow 0}$, then
\[
\lambda_{\mathbb{N}^{''}}[H_{n}, M \cap N]\ \leq\ \omega_{S}(M).
\]
    \item Suppose that pair ${(M,N)}$ is lower semi--Fredholm and define a real number
\[
\omega\ =\ 2 \times (\min(\lambda_{M}, \lambda_{N}) + \varphi(M, N) \times (\lambda_{M} + \lambda_{N})).
\]  
If ${\omega < 1/2}$, then there exists a closed subspace $S$ such that ${M+N = M \cap N \oplus S}$ and 
\[
\omega_{S}(M)\ =\ 2 \times (\lambda_{M} + \varphi_{S}(M, N) \times (\lambda_{M} + \lambda_{N}))\ <\ \frac{1}{2}.
\]
\end{enumerate}
\end{lemma}
\begin{proof}
Proposition $1$ is true due to proposition $2$ from Theorem $2.5.1$ from \cite{burshteyn}. The proof of the proposition $2$ can be found in the proof of the proposition $1$ from the same Theorem $2.5.1$ from \cite{burshteyn}. 
\end{proof}
We are now ready to formulate and prove the lower semi--Fredholm stability theorem for pairs of subspaces:
\begin{theorem}[Lower Semi--Fredholm Pairs are Stable under Singular Adjustment]\label{lsfpssa}
Let $(M,N)$ be a lower semi--Fredholm pair of closed subspaces in a Banach space $X$. Let ${(M_{n})_{\mathbb{N}^{'}}}$, ${(N_{n})_{\mathbb{N}^{'}}}$ are two sequences of closed subspaces from $X$, and set
\[
\mathcal{F}\mathcal{S}\mathcal{S}\lambda_{M} = \mathcal{F}\mathcal{S}\mathcal{S}\lambda_{\mathbb{N}^{'}}[M_{n}, M],\ \ \ \mathcal{F}\mathcal{S}\mathcal{S}\lambda_{N} = \mathcal{F}\mathcal{S}\mathcal{S}\lambda_{\mathbb{N}^{'}}[N_{n}, N];
\]
\[
\mathcal{S}\mathcal{S}\lambda_{M} = \mathcal{S}\mathcal{S}\lambda_{\mathbb{N}^{'}}[M_{n}, M],\ \ \ \mathcal{S}\mathcal{S}\lambda_{N} = \mathcal{S}\mathcal{S}\lambda_{\mathbb{N}^{'}}[N_{n}, N];
\]
then the following propositions are true
\begin{enumerate}
  \item Define a real number
\[
\mathcal{F}\mathcal{S}\mathcal{S}\omega\ =\ 2 \times (\min(\mathcal{F}\mathcal{S}\mathcal{S}\lambda_{M}, \mathcal{F}\mathcal{S}\mathcal{S}\lambda_{N}) + \varphi(M, N) \times (\mathcal{F}\mathcal{S}\mathcal{S}\lambda_{M} + \mathcal{F}\mathcal{S}\mathcal{S}\lambda_{N})).
\]  
If ${\mathcal{F}\mathcal{S}\mathcal{S}\omega < 1/2}$, then for large enough ${n \in \mathbb{N}^{'}}$ pairs ${(M_{n},N_{n})}$ are also lower semi--Fredholm and
\[
\varlimsup_{n \in\mathbb{N}^{'}} \alpha (M_{n}, N_{n})\ <\ \infty.
\]
  \item Define a real number
\[
\mathcal{S}\mathcal{S}\omega\ =\ 2 \times (\min(\mathcal{S}\mathcal{S}\lambda_{M}, \mathcal{S}\mathcal{S}\lambda_{N}) + \varphi(M, N) \times (\mathcal{S}\mathcal{S}\lambda_{M} + \mathcal{S}\mathcal{S}\lambda_{N})).
\]  
If ${\mathcal{S}\mathcal{S}\omega < 1/2}$, then for large enough ${n \in \mathbb{N}^{'}}$ pairs ${(M_{n},N_{n})}$ are also lower semi--Fredholm.
\end{enumerate}
\end{theorem}
\begin{proof}
In order to prove proposition $1$ we first prove that for large enough ${n \in \mathbb{N}^{'}}$ defect numbers ${\alpha(M_{n}, N_{n}) = \dim M_{n} \cap N_{n}}$ are finite and limited from above. In order to do that assume the opposite -- then there exists a subsequence ${\mathbb{N}^{''} \subset \mathbb{N}^{'}}$ such that ${(\dim M_{n} \cap N_{n})_{n \in \mathbb{N}^{''}} \rightarrow \infty}$. According to the condition of the theorem we can choose a real number ${\varepsilon > 0}$ such that
\begin{multline}\label{E:10}
\begin{aligned}
0\ <\ \varepsilon\ <\ \frac{1 - 2 \times \mathcal{F}\mathcal{S}\mathcal{S}\omega}{4 \times (1 + 2 \varphi(M,N))};
\end{aligned}
\end{multline}
since ${M_{n} \cap N_{n} \subset M_{n}}$ for each ${n \in \mathbb{N}^{''}}$, by definition of finitely strictly singular uniform $\lambda-$adjustment there exists a subsequence of finite-dimensional subspaces ${(G_{n})_{\mathbb{N}{'''}} \prec (M_{n} \cap N_{n})_{\mathbb{N}^{''}}}$ such that\linebreak ${(\dim G_{n})_{n \in \mathbb{N}^{''}} \rightarrow \infty}$ and that ${\lambda_{\mathbb{N}^{'''}}[G_{n}, M] \leq \mathcal{F}\mathcal{S}\mathcal{S}\lambda_{\mathbb{N}^{'}}[M_{n}, M] + \varepsilon}$. Since ${G_{n} \subset M_{n} \cap N_{n} \subset N_{n}}$ for each ${n \in \mathbb{N}^{'''}}$, again by definition of finitely strictly singular uniform $\lambda-$adjustment there exists a subsequence of subspaces ${(H_{n})_{\mathbb{N}{'4}} \prec (G_{n})_{\mathbb{N}^{'4}}}$ such that ${(\dim H_{n})_{n \in \mathbb{N}^{'4}} \rightarrow \infty}$ and that
\[
\lambda_{N}\ :=\ \lambda_{\mathbb{N}^{'4}}[H_{n}, N]\ \leq\  \mathcal{F}\mathcal{S}\mathcal{S}\lambda_{\mathbb{N}^{'}}[N_{n}, N] + \varepsilon\ =\ \mathcal{F}\mathcal{S}\mathcal{S}\lambda_{N} + \varepsilon.
\]
At the same time, since ${(H_{n})_{\mathbb{N}{'4}} \prec (G_{n})_{\mathbb{N}^{'4}}}$ and by the choice of ${(G_{n})_{\mathbb{N}^{'''}}}$ we can write
\[
\lambda_{M}\ :=\ \lambda_{\mathbb{N}^{'4}}[H_{n}, M]\ \leq\ \lambda_{\mathbb{N}^{'''}}[G_{n}, M]\ \leq\  \mathcal{F}\mathcal{S}\mathcal{S}\lambda_{\mathbb{N}^{'}}[M_{n}, M] + \varepsilon\ =\ \mathcal{F}\mathcal{S}\mathcal{S}\lambda_{M} + \varepsilon.
\]
From the above two inequalities, from the choice of $\varepsilon$ and from the condition of the theorem we can estimate
\begin{multline}
\notag
\begin{aligned}
\omega\ &=\ 2 \times (\min(\lambda_{M}, \lambda_{N}) + \varphi(M, N) \times (\lambda_{M} + \lambda_{N}))\\
               &\leq\ 2 \times (\min(\mathcal{F}\mathcal{S}\mathcal{S}\lambda_{M} + \varepsilon, \mathcal{F}\mathcal{S}\mathcal{S}\lambda_{N} + \varepsilon) + \varphi(M, N) \times (\min(\mathcal{F}\mathcal{S}\mathcal{S}\lambda_{M} + \varepsilon + \mathcal{F}\mathcal{S}\mathcal{S}\lambda_{N} + \varepsilon))\\
               &=\ 2 \times (\min(\mathcal{F}\mathcal{S}\mathcal{S}\lambda_{M}, \mathcal{F}\mathcal{S}\mathcal{S}\lambda_{N}) + \varepsilon + \varphi(M, N) \times (\mathcal{F}\mathcal{S}\mathcal{S}\lambda_{M} + \mathcal{F}\mathcal{S}\mathcal{S}\lambda_{N}) + 2  \varphi(M, N) \varepsilon)\\
               &=\ 2 \times (\min(\mathcal{F}\mathcal{S}\mathcal{S}\lambda_{M}, \mathcal{F}\mathcal{S}\mathcal{S}\lambda_{N}) + \varphi(M, N) \times (\mathcal{F}\mathcal{S}\mathcal{S}\lambda_{M} + \mathcal{F}\mathcal{S}\mathcal{S}\lambda_{N})) + 2 \varepsilon (1 + 2 \varphi(M, N))\\
               &=\ \mathcal{F}\mathcal{S}\mathcal{S}\omega +  2 \varepsilon (1 + 2 \varphi(M, N))\ <\ \mathcal{F}\mathcal{S}\mathcal{S}\omega + 2 \times \frac{1 - 2 \times \mathcal{F}\mathcal{S}\mathcal{S}\omega}{4 \times (1 + 2 \varphi(M,N))} \times (1 + 2 \varphi(M, N))\\
               &=\ \mathcal{F}\mathcal{S}\mathcal{S}\omega + \frac{1}{2} - \mathcal{F}\mathcal{S}\mathcal{S}\omega\ =\ \frac{1}{2}.
\end{aligned}
\end{multline}
Therefore, we may now apply proposition $2$ from the previous lemma \ref{ftl} in case of the pair of sequences where each sequence is the same ${(H_{n})_{\mathbb{N}^{'4}}}$ and find a closed subspace ${S \subset X}$ such that ${X = M \cap N \oplus S}$ and that 
\[
\omega_{S}(M)\ =\ 2 \times (\lambda_{M} + \varphi_{S}(M, N) \times (\lambda_{M} + \lambda_{N}))\ <\ \frac{1}{2}.
\]
Due to above inequality and since obviously ${\delta(H_{n}, H_{n}) = 0}$, we may now apply proposition $1$ of the same lemma \ref{ftl} and conclude that
\[
\lambda_{\mathbb{N}^{'4}}[H_{n}, M \cap N]\ <\ \frac{1}{2}.
\]
Therefore, since ${\dim M \cap N < \infty}$, by the small uniform adjustment theorem $2.2.2$ from \cite{burshteyn} we conclude that dimensions of all subspaces ${H_{n}}$ are limited from above by a finite number which contradicts our choice of ${\dim H_{n} \rightarrow \infty}$. Therefore our assumption that there exists a subsequence ${\mathbb{N}^{''} \subset \mathbb{N}^{'}}$ such that ${(\dim M_{n} \cap N_{n})_{n \in \mathbb{N}^{''}} \rightarrow \infty}$ is incorrect and we shall conclude that for large enough ${n \in \mathbb{N}^{'}}$ the defect numbers ${\alpha(M_{n},N_{n})}$ which are the dimensions of subspaces ${M_{n} \cap N_{n}}$ are limited from above by a finite number.

In order to conclude the proof of the proposition $1$ we shall establish that ${M_{n} + N_{n} = \overline{M_{n} + N_{n}}}$ for large enough ${n \in \mathbb{N}^{'}}$. In order to do that assume the opposite and choose a subsequence ${\mathbb{N}^{''} \subset \mathbb{N}^{'}}$ such that ${M_{n} + N_{n} \neq \overline{M_{n} + N_{n}}}$ for each ${n \in \mathbb{N}^{''}}$. Then, we are going to build a number of sequences of subspaces through the following steps:
\begin{enumerate}
  \item According to theorem $4.19$ from \cite{kato}, page 226, there exist infinite-dimensional subspaces ${H_{n} \subset M_{n}}$ such that ${\delta(H_{n}, N_{n}) < n^{-1}}$. Therefore, we can choose a sequence of finite-dimensional subspaces ${(G_{n})_{\mathbb{N}^{''}} \prec (H_{n})_{\mathbb{N}^{''}}}$ such that ${(\dim G_{n})_{\mathbb{N}^{''}} \rightarrow \infty}$ and ${\delta(G_{n}, N_{n})_{\mathbb{N}^{''}} \rightarrow 0}$. 
  \item Now choose a real number ${\varepsilon}$ as in formula \eqref{E:10} above and recall that by definition of finitely strictly singular uniform $\lambda-$adjustment there exists a subsequence of finite-dimensional subspaces ${(K_{n})_{\mathbb{N}{'''}} \prec (G_{n})_{\mathbb{N}^{''}}}$ such that ${(\dim K_{n})_{n \in \mathbb{N}^{'''}} \rightarrow \infty}$ and that
  \begin{multline}\label{E:11}
  \begin{aligned}
  \lambda_{\mathbb{N}^{'''}}[K_{n}, M]\ \leq\  \mathcal{F}\mathcal{S}\mathcal{S}\lambda_{\mathbb{N}^{'}}[M_{n}, M] + \varepsilon.
  \end{aligned}
  \end{multline}
  \item Obviously ${\delta(K_{n}, N_{n})_{\mathbb{N}^{'''}} \rightarrow 0}$, hence by the previous lemma \ref{253} there exist two subsequences of subspaces ${(L_{n})_{\mathbb{N}^{'4}_{1}} \prec (K_{n})_{\mathbb{N}^{'4}}}$ and ${(P_{n})_{\mathbb{N}^{'4}} \prec (N_{n})_{\mathbb{N}^{'4}}}$ such that ${\dim L_{n} = \dim P_{n}}$ for all ${n \in \mathbb{N}^{'4}}$, ${\lim_{n \in \mathbb{N}^{'4}} \dim L_{n} = \infty}$ and ${\lim_{n \in \mathbb{N}^{'4}} \delta(L_{n}, P_{n}) = \lim_{n \in \mathbb{N}^{'4}} \delta(P_{n}, L_{n}) = 0}$. 
  \item Again by definition of finitely strictly singular uniform $\lambda-$adjustment there exists a subsequence of finite-dimensional subspaces ${(Q_{n})_{\mathbb{N}{'5}} \prec (P_{n})_{\mathbb{N}^{'5}}}$ such that ${(\dim Q_{n})_{n \in \mathbb{N}^{'5}} \rightarrow \infty}$ and that
  \begin{multline}\label{E:12}
  \begin{aligned}
\lambda_{\mathbb{N}^{'5}}[Q_{n}, N]\ \leq\ \mathcal{F}\mathcal{S}\mathcal{S}\lambda_{\mathbb{N}^{'}}[N_{n}, N] + \varepsilon.
  \end{aligned}
  \end{multline}
  \item Obviously ${\lim_{n \in \mathbb{N}^{'5}} \delta(Q_{n}, L_{n}) = \lim_{n \in \mathbb{N}^{'5}} \delta(P_{n}, L_{n}) = 0}$ since ${(Q_{n})_{\mathbb{N}{'5}} \prec (P_{n})_{\mathbb{N}^{'5}}}$; therefore applying the previous lemma \ref{253} this time to the pair of sequences of subspaces ${(Q_{n}, L_{n})_{\mathbb{N}^{'5}}}$ we can find two subsequences of subspaces ${(R_{n})_{\mathbb{N}{'6}} \prec (L_{n})_{\mathbb{N}^{'6}}}$ and ${(T_{n})_{\mathbb{N}{'6}} \prec (Q_{n})_{\mathbb{N}^{'6}}}$ such that ${\dim R_{n} = \dim T_{n}}$ for all ${n \in \mathbb{N}^{'6}}$, ${\lim_{n \in \mathbb{N}^{'6}} \dim R_{n} = \infty}$ and
  \begin{multline}\label{E:14}
  \begin{aligned}
  \lim_{n \in \mathbb{N}^{'6}} \delta(R_{n}, T_{n})\ =\ \lim_{n \in \mathbb{N}^{'6}} \delta(T_{n}, R_{n})\  =\ 0. 
  \end{aligned}
  \end{multline}
\end{enumerate}
The following two lines summarize the inclusion relations between the subspaces in our constructions:
\begin{multline}
\notag
\begin{aligned}
&(R_{n})_{\mathbb{N}^{'6}} & &\ \prec (L_{n})_{\mathbb{N}^{'4}} &\ \prec (K_{n})_{\mathbb{N}^{'''}} &\ \prec (G_{n})_{\mathbb{N}^{''}} &\ \prec (H_{n})_{\mathbb{N}^{''}} &\ \prec (M_{n})_{\mathbb{N}^{'}}\\
&(T_{n})_{\mathbb{N}^{'6}} &\ \prec (Q_{n})_{\mathbb{N}^{'5}} &\ \prec (P_{n})_{\mathbb{N}^{'4}} & & & &\ \prec (N_{n})_{\mathbb{N}^{'}}
\end{aligned}
\end{multline}
Combining the above relations with the previous estimates \ref{E:11} and \ref{E:12} we come to the estimates
\begin{multline}
\notag
\begin{aligned}
&\lambda_{M} &:=\ \lambda_{\mathbb{N}^{'6}}[R_{n}, M]\ &\leq\ &\lambda_{\mathbb{N}^{'''}}[K_{n}, M]\ &\leq\  \mathcal{F}\mathcal{S}\mathcal{S}\lambda_{\mathbb{N}^{'}}[M_{n}, M] + \varepsilon &=\ \mathcal{F}\mathcal{S}\mathcal{S}_{M} + \varepsilon,\\ 
&\lambda_{N} &:=\ \lambda_{\mathbb{N}^{'6}}[T_{n}, N]\ &\leq\ &\lambda_{\mathbb{N}^{'5}}[Q_{n}, N]\ &\leq\  \mathcal{F}\mathcal{S}\mathcal{S}\lambda_{\mathbb{N}^{'}}[N_{n}, N] + \varepsilon &=\ \mathcal{F}\mathcal{S}\mathcal{S}_{N} + \varepsilon.
\end{aligned}
\end{multline}
Now define
\[
\omega\ =\ 2 \times (\min(\lambda_{M}, \lambda_{N}) + \varphi(M, N) \times (\lambda_{M} + \lambda_{N})).
\]
According to our choice of $\varepsilon$ from \eqref{E:10} we can estimate ${\omega < 1/2}$ as before. Therefore, we can apply proposition $2$ from the previous lemma \ref{ftl} to two sequences of subspaces ${(R_{n})_{\mathbb{N}^{'6}}}$ and ${(T_{n})_{\mathbb{N}^{'6}}}$ and choose a closed subspace ${S \subset X}$ such that ${M+N = M \cap N \oplus S}$ and 
\[
\omega_{S}(M)\ =\ 2 \times (\lambda_{M} + \varphi_{S}(M, N) \times (\lambda_{M} + \lambda_{N}))\ <\ \frac{1}{2}.
\]
Now taking into account the gap estimate from \eqref{E:14} we may apply proposition $1$ from the same lemma \ref{ftl} and conclude that
\[
\lambda_{\mathbb{N}^{'6}}[R_{n}, M]\ \leq\ \omega_{S}(M).
\]
Therefore ${\lambda_{\mathbb{N}^{'6}}[R_{n}, M] < 1/2}$, hence since ${\dim M \cap N < \infty}$, by the small uniform adjustment theorem $2.2.2$ from \cite{burshteyn} we conclude that dimensions of all subspaces ${R_{n}}$ are limited from above by a finite number for large enough ${n \in \mathbb{N}^{'6}}$ which contradicts our construction in step $5$ as ${\dim R_{n} \rightarrow \infty}$. Therefore the assumption that there exists subsequence ${\mathbb{N}^{''} \subset \mathbb{N}^{'}}$ such that ${M_{n} + N_{n} \neq \overline{M_{n} + N_{n}}}$ for each ${n \in \mathbb{N}^{''}}$ is incorrect which means that ${M_{n} + N_{n} = \overline{M_{n} + N_{n}}}$ for large enough ${n \in \mathbb{N}^{'}}$ which concludes the the proof of the theorem's proposition $1$.
\\

We are now commencing the proof of the proposition 2. First we establish that ${\dim M_{n} \cap N_{n} < \infty}$ for large enough ${n \in \mathbb{N}^{'}}$. In order to do that assume the opposite -- then there exists a subsequence ${\mathbb{N}^{''}}$ such that ${\dim M_{n} \cap N_{n} = \infty}$ for all ${n \in \mathbb{N}^{''}}$. Hence, if ${\varepsilon}$ is a real number such that
\begin{multline}\label{E:1000}
\begin{aligned}
0\ <\ \varepsilon\ <\ \frac{1 - 2 \times \mathcal{S}\mathcal{S}\omega}{4 \times (1 + 2 \varphi(M,N))},
\end{aligned}
\end{multline} 
then since ${M_{n} \cap N_{n} \subset M_{n}}$ for each ${n \in \mathbb{N}^{''}}$, by definition of strictly singular uniform $\lambda-$adjustment there exists a subsequence of infinite-dimensional subspaces ${(G_{n})_{\mathbb{N}{'''}} \prec (M_{n} \cap N_{n})_{\mathbb{N}^{''}}}$ such that  ${\lambda_{\mathbb{N}^{'''}}[G_{n}, M] \leq \mathcal{S}\mathcal{S}\lambda_{\mathbb{N}^{'}}[M_{n}, M] + \varepsilon}$. Since ${G_{n} \subset M_{n} \cap N_{n} \subset N_{n}}$ for each ${n \in \mathbb{N}^{'''}}$, again by definition of finitely singular uniform $\lambda-$adjustment there exists a subsequence of subspaces\linebreak ${(H_{n})_{\mathbb{N}{'4}} \prec (G_{n})_{\mathbb{N}^{'4}}}$ such that ${(\dim H_{n})_{n \in \mathbb{N}^{'4}} = \infty}$ and that
\[
\lambda_{N}\ :=\ \lambda_{\mathbb{N}^{'4}}[H_{n}, N]\ \leq\  \mathcal{S}\mathcal{S}\lambda_{\mathbb{N}^{'}}[N_{n}, N] + \varepsilon\ =\ \mathcal{S}\mathcal{S}\lambda_{N} + \varepsilon.
\]
At the same time, since ${(H_{n})_{\mathbb{N}{'4}} \prec (G_{n})_{\mathbb{N}^{'4}}}$ and by the choice of ${(G_{n})_{\mathbb{N}^{'''}}}$ we can write
\[
\lambda_{M}\ :=\ \lambda_{\mathbb{N}^{'4}}[H_{n}, M]\ \leq\ \lambda_{\mathbb{N}^{'''}}[G_{n}, M]\ \leq\  \mathcal{S}\mathcal{S}\lambda_{\mathbb{N}^{'}}[M_{n}, M] + \varepsilon\ =\ \mathcal{S}\mathcal{S}\lambda_{M} + \varepsilon.
\]
From this point we may continue the proof exactly as in the previous proof for finitely strictly singular case and come to a contradiction. Therefore, we must conclude that that ${\dim M_{n} \cap N_{n} < \infty}$ for large enough ${n \in \mathbb{N}^{'}}$.

We are now left to prove that ${M_{n} + N_{n} = \overline{M_{n} + N_{n}}}$ for large enough ${n \in \mathbb{N}^{'}}$. First note that since ${\dim M_{n} \cap N_{n} < \infty}$ for large enough ${n \in \mathbb{N}^{'}}$ there exist closed subspaces ${S_{n} \subset X}$ such that ${(M_{n} \cap N_{n}) \oplus S_{n} = X}$ for the same $n$. Therefore, it is enough to prove that\linebreak ${(M_{n} \cap S_{n}) \oplus (N_{n} \cap S_{n}) = \overline{(M_{n} \cap S_{n}) \oplus (N_{n} \cap S_{n})}}$ for large enough ${n \in \mathbb{N}^{'}}$. Thus, since\linebreak ${(M_{n} \cap S_{n}) \cap (N_{n} \cap S_{n}) = \{\theta\}}$, from now on we may safely assume that ${M_{n} \cap N_{n} = \{\theta\}}$.

Now assume the opposite -- then there exists a subsequence ${\mathbb{N}^{''} \subset \mathbb{N}^{'}}$ such that\linebreak ${M_{n} + N_{n} \neq \overline{M_{n} + N_{n}}}$ for all ${n \in \mathbb{N}^{''}}$. Then proceed building sequences of closed infinite-dimensional subspaces in the following steps:
\begin{enumerate}
  \item For each such $n$ choose two closed subspaces ${H_{n} \subset M_{n}}$ and ${G_{n} \subset N_{n}}$ as in lemma \ref{hncsip} such that ${\max(\delta(H_{n}, G_{n}), \delta(G_{n}, H_{n})) < n^{-1}}$ and that ${H_{n} + G_{n} \neq \overline{H_{n} + G_{n}}}$. 
  \item Choose a real number ${\varepsilon}$ as in formula \eqref{E:1000} above and recall that by definition of strictly singular uniform $\lambda-$adjustment there exists a subsequence of infinite-dimensional subspaces ${(K_{n})_{\mathbb{N}{'''}} \prec (H_{n})_{\mathbb{N}^{''}}}$ such that
  \begin{multline}\label{E:111}
  \begin{aligned}
  \lambda_{\mathbb{N}^{'''}}[K_{n}, M]\ \leq\ \mathcal{S}\mathcal{S}\lambda_{\mathbb{N}^{'}}[M_{n}, M] + \varepsilon.
  \end{aligned}
  \end{multline}
  \item By lemma \ref{hncsip} there exists a sequence of infinite-dimensional subspaces ${(L_{n})_{\mathbb{N}^{'''}} \prec (G_{n})_{\mathbb{N}^{'''}}}$ such that ${\max(\delta(K_{n}, L_{n}), \delta(L_{n}, K_{n})) < n^{-1}}$ and that ${K_{n} + L_{n} \neq \overline{K_{n} + L_{n}}}$ for each ${n \in \mathbb{N}^{'''}}$.
  \item Again by definition of finitely singular uniform $\lambda-$adjustment there exists a subsequence of infinite-dimensional subspaces ${(P_{n})_{\mathbb{N}{'4}} \prec (L_{n})_{\mathbb{N}^{'''}}}$ such that
  \begin{multline}
  \notag
  \begin{aligned}
  \lambda_{N}\ :=\ \lambda_{\mathbb{N}^{'4}}[P_{n}, N]\ \leq\ \mathcal{S}\mathcal{S}\lambda_{\mathbb{N}^{'}}[N_{n}, N] + \varepsilon.
  \end{aligned}
  \end{multline}
  \item Then by lemma \ref{hncsip} there exists an infinite-dimensional subspace ${Q_{n} \subset K_{n}}$ such that ${\max(\delta(Q_{n}, P_{n}), \delta(P_{n}, Q_{n})) < n^{-1}}$ and that ${Q_{n} + P_{n} \neq \overline{Q_{n} + P_{n}}}$ for each ${n \in \mathbb{N}^{'4}}$. From the last inclusion ${Q_{n} \subset K_{n}}$ and from inequality \eqref{E:111} it follows that
  \begin{multline}
  \notag
  \begin{aligned}
  \lambda_{M}\ :=\ \lambda_{\mathbb{N}^{'4}}[Q_{n}, M]\ \leq\ \lambda_{\mathbb{N}^{'''}}[K_{n}, M]\ \leq\  \mathcal{S}\mathcal{S}\lambda_{\mathbb{N}^{'}}[M_{n}, M] + \varepsilon.
  \end{aligned}
  \end{multline}\end{enumerate}
From this point we may follow the final part of the proof of the proposition $1$ applied to sequences of subspaces ${(P_{n})_{\mathbb{N}^{'4}}}$ and ${(Q_{n})_{\mathbb{N}^{'4}}}$ and conclude that for large enough ${n \in \mathbb{N}^{'4}}$ dimensions of all spaces ${Q_{n}}$ are finite and limited from above which contradicts our choice in step $5$ where every subspace ${Q_{n}}$ is infinite-dimensional. This concludes the proof of the proposition $2$ and of the entire theorem.
\end{proof}
An immediate consequence from the above Theorem \ref{lsfpssa} is the next theorem on stability of lower semi-Fredholm operators:
\begin{theorem}[Lower Semi-Fredholm operators are stable under small strictly singular adjustment]\label{lsfoasusmsa}
Recall that if $X$, $Y$ are two Banach spaces and ${A \in \mathcal{C}(X,Y)}$ (i.e. $A$ is a closed operator from $X$ to $Y$), then $A$ is called \emph{a lower semi-Fredholm operator} if its image ${R(A)}$ is closed in $Y$ and dimension of its kernel ${Ker(A)}$ in $X$ is finite. Let ${(A_{n})_{\mathbb{N}^{'}} \subset \mathcal{C}(X,Y)}$ be a sequence of closed operators from $X$ to $Y$. Then the following propositions are true:from $X$ to $Y$. Then the following propositions are true:
\begin{itemize} 
\item If ${\mathcal{F}\mathcal{S}\mathcal{S}\lambda_{\mathbb{N}^{'}}[A_{n}, A]}$ is small enough, then for large enough ${n \in \mathbb{N}^{'}}$ operators ${A_{n}}$ are also lower semi-Fredholm and ${\varlimsup \dim Ker(A_{n}) < \infty}$.
\item If ${\mathcal{S}\mathcal{S}\lambda_{\mathbb{N}^{'}}[A_{n}, A]}$ is small enough, then for large enough ${n \in \mathbb{N}^{'}}$ operators ${A_{n}}$ are also lower semi-Fredholm.
\end{itemize}
\end{theorem}
\begin{proof}
By the standard procedure from \cite{kato} we reduce lower semi-Fredholm operators to lower semi-Fredholm pairs of subspaces as follows. First we consider a product space ${X \times Y}$ and a pair of its subspaces ${M := X \times \{\theta\}}$ and ${N := G_{A}}$. Then for each closed operator $A_{n}$ define a closed subspace ${N_{n} := G_{A_{n}}}$. It is clear that the pair of subspaces ${(M , N)}$ from ${X \times Y}$ is lower semi-Fredholm since operator $A$ is lower semi-Fredholm and that ${\dim M \cap N = \dim Ker(A)}$. Also by definitions \ref{D:ssulssclo} and \ref{D:nso}
\begin{multline}
\notag
\begin{aligned}
\mathcal{F}\mathcal{S}\mathcal{S}\lambda_{\mathbb{N}^{'}}[A_{n}, A]\ &=\  \mathcal{F}\mathcal{S}\mathcal{S}\lambda_{\mathbb{N}^{'}}[G_{A_{n}}, G_{A}]\ =\  \mathcal{F}\mathcal{S}\mathcal{S}\lambda_{\mathbb{N}^{'}}[N_{n}, N],\\ \mathcal{S}\mathcal{S}\lambda_{\mathbb{N}^{'}}[A_{n}, A]\ &=\  \mathcal{S}\mathcal{S}\lambda_{\mathbb{N}^{'}}[G_{A_{n}}, G_{A}]\ =\  \mathcal{S}\mathcal{S}\lambda_{\mathbb{N}^{'}}[N_{n}, N].
\end{aligned}
\end{multline}
Therefore, by the previous Theorem \ref{lsfpssa}, if either number ${\mathcal{F}\mathcal{S}\mathcal{S}\lambda_{\mathbb{N}^{'}}[A_{n}, A]}$ or ${\mathcal{S}\mathcal{S}\lambda_{\mathbb{N}^{'}}[A_{n}, A]}$ is small enough, then pairs of subspaces ${(M, N_{n})}$ are also lower lower semi-Fredholm for large enough ${n \in \mathbb{N}^{'}}$. Thus for the same $n$ operators ${A_{n}}$ are lower semi-Fredholm as well. Also, if ${\mathcal{F}\mathcal{S}\mathcal{S}\lambda_{\mathbb{N}^{'}}[A_{n}, A]}$ is small enough, then ${\varlimsup \dim M \cap N_{n} = \varlimsup \dim Ker(A_{n}) < \infty}$.
\end{proof}
An immediate consequence of the above stability theorem is that the dimensions of kernels under finitely singularly composition perturbations are limited:
\begin{theorem}[Stability of kernels of Composition Operator Sequences]\label{skcos}
Let $X$, $Y$ are Banach spaces with three sequences of linear operators ${(S_{n})_{\mathbb{N}^{'}} \subset \mathcal{B}(Y, Y)}$, ${(B_{n})_{\mathbb{N}^{'}} \subset \mathcal{B}(X, Y)}$ and ${(T_{n})_{\mathbb{N}^{'}} \subset \mathcal{B}(X, X)}$:
\[
\begin{CD}
X @>T_{n}>> X @>B_{n}>> Y @>S_{n}>> Y.
\end{CD}
\]
Let ${A \in \mathcal{B}(X, Y)}$ be a lower semi-Fredholm operator and suppose that ${\varlimsup \left\|S_{n}\right\| < \infty}$, ${\varlimsup \left\|T_{n}\right\| < \infty}$ and ${\mathcal{F}\mathcal{S}\mathcal{S}\lambda_{\mathbb{N}^{'}}[B_{n},\ \theta] = 0}$. Then for large enough ${n \in \mathbb{N}^{'}}$ operators ${A + B_{n} \circ T_{n}}$ and ${A + S_{n} \circ B_{n}}$ are also lower semi-Fredholm and dimensions of their kernels ${Ker(A + B_{n} \circ T_{n})}$, ${Ker(A + S_{n} \circ B_{n})}$ are limited by a finite number. 
\end{theorem}
\begin{proof}
The proof follows immediately from the above Theorem \ref{lsfoasusmsa} after noticing that\linebreak ${\mathcal{F}\mathcal{S}\mathcal{S}\lambda_{\mathbb{N}^{'}}[B_{n} \circ T_{n},\ \theta] = \mathcal{F}\mathcal{S}\mathcal{S}\lambda_{\mathbb{N}^{'}}[S_{n} \circ B_{n},\ \theta] = 0}$ by Lemma \ref{ssco}, and, therefore\linebreak ${\mathcal{F}\mathcal{S}\mathcal{S}\lambda_{\mathbb{N}^{'}}[A + B_{n} \circ T_{n},\ A] = \mathcal{F}\mathcal{S}\mathcal{S}\lambda_{\mathbb{N}^{'}}[A + S_{n} \circ B_{n},\ A] = 0}$ by Remark \ref{r:1}. 
\end{proof}
\section[Conclusion]{Conclusion}\label{c}
In conclusion we consider a weaker variant of strictly singular $\lambda-$adjustment that still allows for lower semi-Fredholm stability and discuss some open problems related to the interplay between strictly singular $\lambda-$adjustment and the geometry of Banach spaces.
\subsection[Relaxed Strict Singularity]{Relaxed Strict Singularity}\label{ssrss}
One can relax definition of number ${\mathcal{S}\mathcal{S}\lambda_{\mathbb{N}^{'}}[M, M_{n}]}$ by allowing finite dimensions of the subspaces ${(K_{n})_{\mathbb{N}^{''}} \prec (L_{n})_{\mathbb{N}^{'}}}$ from Definition \ref{ssua} as follows:
\begin{definition}[Relaxed Strict Singularity]\label{D:rss}
Let $(M_{n})_{\mathbb{N}^{'}}$ and $(P_{n})_{\mathbb{N}^{'}}$ be a pair of sequences of closed subspaces from a Banach space $X$, ${M_{n} \neq \{\theta\}}$ for all ${n \in \mathbb{N}^{'}}$ and $\lambda \geq 0$. We say that $(M_{n})_{\mathbb{N}^{'}}$ is $lower$\ $relaxed$\ $strictly$\ $singular$\ $uniformly$\ $\lambda-adjusted$ with $(P_{n})_{\mathbb{N}^{'}}$ if for any subsequence of closed subspaces $(K_{n})_{\mathbb{N}^{''}} \prec (M_{n})_{\mathbb{N}^{''}}$ such that ${\dim K_{n} = \infty}$ for all ${n \in \mathbb{N''}}$ there exists a subsequence of closed subspaces $(L_{n})_{\mathbb{N}^{'''}} \prec (K_{n})_{\mathbb{N}^{'''}}$ such that ${\dim L_{n} \rightarrow \infty}$ with the property ${\lambda_{\mathbb{N}^{'''}}[L_{n}, P_{n}] \leq \lambda}$. Let $\mathcal{R}\mathcal{S}\mathcal{S}\Lambda_{\mathbb{N}^{'}}[M_{n}, P_{n}]$ be the set of all such real numbers $\lambda$; then the $relaxed$\ $strictly$\ $singular$\ $uniform$\ $\lambda-$$adjustment$\ between $(M_{n})_{\mathbb{N}^{'}}$ and $(P_{n})_{\mathbb{N}^{'}}$ is a non-negative real number defined as
\[
\mathcal{R}\mathcal{S}\mathcal{S}\lambda_{\mathbb{N}^{'}}[M_{n}, P_{n}]\ :=\ \inf \{ \lambda \in \mathcal{R}\mathcal{S}\mathcal{S}\Lambda_{\mathbb{N}^{'}}[M_{n}, P_{n}] \}.
\]
\end{definition}
One can also extend this definition to closed operators as before by considering operator's graphs in the product space. Then with the method used in the proof of the previous theorems \ref{lsfpssa} and \ref{lsfoasusmsa} it is possible to prove the following version of semi-Fredholm stability (we omit its proof as it is almost the same as the proof of the previous theorems):
\begin{theorem} Let $X$ and $Y$ be two Banach spaces. Then the following propositions are true:
\begin{itemize}
\item Let $(M,N)$ be a lower semi--Fredholm pair of closed subspaces in a Banach space $X$. Let ${(M_{n})_{\mathbb{N}^{'}}}$, ${(N_{n})_{\mathbb{N}^{'}}}$ are two sequences of closed subspaces from $X$ such that both numbers\linebreak ${\mathcal{F}\mathcal{S}\mathcal{S}\lambda_{\mathbb{N}^{'}}[M_{n}, M]}$ and ${\mathcal{R}\mathcal{S}\mathcal{S}\lambda_{\mathbb{N}^{'}}[N_{n}, N]}$ are small enough. Then pairs of subspaces ${(M_{n}, N_{n})}$ are also lower semi-Fredholm for large enough ${n \in \mathbb{N}^{'}}$.
\item Let ${A \in \mathcal{C}(X,Y)}$ be a lower semi-Fredholm operator and ${(A_{n})_{\mathbb{N}^{'}} \subset \mathcal{C}(X,Y)}$ be a sequence of closed operators from $X$ to $Y$. If number ${\mathcal{R}\mathcal{S}\mathcal{S}\lambda_{\mathbb{N}^{'}}[A_{n}, A]}$ is small enough, then for large enough ${n \in \mathbb{N}^{'}}$ operators ${A_{n}}$ are also lower semi-Fredholm.
\end{itemize}
\end{theorem}
\subsection[Strict Singularity and the Geometry of Banach Spaces]{Strict Singularity and Geometry of Banach Spaces}\label{tgoss}
The next lemma is an extension of the Small Gap Theorem \ref{tsgp} to strictly singular adjustment:
\begin{lemma}[The Small Strictly Singular Adjustment Theorem]\label{tsssat}
Let ${(M_{n})_{\mathbb{N}^{'}}}$ and ${(P_{n})_{\mathbb{N}^{'}}}$ be two sequences of closed subspaces from a Banach space $X$. Then the following propositions are true:
\begin{enumerate}
\item Suppose that ${\mathcal{F}\mathcal{S}\mathcal{S}\lambda_{\mathbb{N}^{'}}[M_{n}, P_{n}]\ <\ \frac{1}{2}}$ and that for large enough ${n \in \mathbb{N}^{'}}$ dimensions of spaces ${P_{n}}$ are finite and limited from above, i.e. ${\varlimsup_{n \in \mathbb{N}^{'}} \dim P_{n} < \infty}$. Then for large enough ${n \in \mathbb{N}^{'}}$ dimensions of spaces ${M_{n}}$ are also finite and limited from above, i.e. ${\varlimsup_{n \in \mathbb{N}^{'}} \dim M_{n} < \infty}$
\item Suppose that ${\mathcal{S}\mathcal{S}\lambda_{\mathbb{N}^{'}}[M_{n}, P_{n}]\ <\ \frac{1}{2}}$ and that for large enough ${n \in \mathbb{N}^{'}}$ dimensions of spaces ${P_{n}}$ are finite. Then for large enough ${n \in \mathbb{N}^{'}}$ dimensions of spaces ${M_{n}}$ are also finite.
\end{enumerate}
\end{lemma}
\begin{proof}
In order to prove proposition $1$ choose ${\delta \in (0, \frac{1}{2} - \mathcal{F}\mathcal{S}\mathcal{S}\lambda_{\mathbb{N}^{'}}[M_{n}, P_{n}])}$ and assume the opposite -- then there exists a subsequence of subspaces ${(M_{n})_{\mathbb{N}^{''}}}$ such that ${\dim M_{n} \rightarrow \infty}$. Therefore, we may choose a subsequence of finite-dimensional subspaces ${(K_{n})_{\mathbb{N}^{''}} \prec (M_{n})_{\mathbb{N}^{''}}}$ such that ${\dim K_{n} \rightarrow \infty}$. Therefore, since ${\mathcal{F}\mathcal{S}\mathcal{S}\lambda_{\mathbb{N}^{'}}[M_{n}, P_{n}]\ <\ \frac{1}{2}}$, there exists a subsequence of finite-\linebreak dimensional subspaces ${(L_{n})_{\mathbb{N}^{'''}} \prec (K_{n})_{\mathbb{N}^{'''}}}$ such that ${\dim L_{n} \rightarrow \infty}$ and that
\[
\lambda_{\mathbb{N}^{'}}[L_{n}, P_{n}]\ <\ \mathcal{F}\mathcal{S}\mathcal{S}\lambda_{\mathbb{N}^{'}}[M_{n}, P_{n}] + \delta\ <\ \frac{1}{2}.
\]
However, by the small uniform adjustment theorem $2.2.2$ from \cite{burshteyn} the last inequality implies that for large enough ${n \in \mathbb{N}^{'''}}$ dimensions of all ${L_{n}}$ are finite and limited from above -- this contradiction means that our assumption is incorrect and therefore proposition $1$ is true.

Proposition $2$ can be proved in a similar way. That is, choose ${\delta \in [0, \frac{1}{2} - \mathcal{S}\mathcal{S}\lambda_{\mathbb{N}^{'}}[M_{n}, P_{n}]]}$ and assume the opposite -- then there exists a subsequence of subspaces ${(M_{n})_{\mathbb{N}^{''}}}$ such that ${\dim M_{n} = \infty}$ for all ${n \in \mathbb{N}^{''}}$. Therefore, since ${\mathcal{S}\mathcal{S}\lambda_{\mathbb{N}^{'}}[M_{n}, P_{n}]\ <\ \frac{1}{2}}$, there exists a subsequence of infinite-dimensional subspaces ${(L_{n})_{\mathbb{N}^{'''}} \prec (K_{n})_{\mathbb{N}^{'''}}}$ such that ${\lambda_{\mathbb{N}^{'}}[K_{n}, P_{n}] < \mathcal{S}\mathcal{S}\lambda_{\mathbb{N}^{'}}[M_{n}, P_{n}] + \delta < \frac{1}{2}}$. However, by the small uniform adjustment theorem $2.2.2$ from \cite{burshteyn} the last inequality implies that for large enough ${n \in \mathbb{N}^{''}}$ dimensions of all ${K_{n}}$ are finite -- this contradiction means that our assumption is incorrect and therefore proposition $2$ is true.
\end{proof}
Note that proposition $1$ from the above theorem assumes that dimensions of all subspaces ${P_{n}}$ are limited by a finite number. This is in contrast to the small uniform adjustment theorem $2.2.2$ from \cite{burshteyn} that does not have this limitation. This is because we do not know if the following stronger version of the proposition $1$, similar to theorem $2.2.2$ from \cite{burshteyn}, is true:
\begin{theorem}[The Strong Small Strictly Singular Adjustment Theorem]\label{tssssat}
Let ${(M_{n})_{\mathbb{N}^{'}}}$ and ${(P_{n})_{\mathbb{N}^{'}}}$ be two sequences of closed subspaces from a Banach space $X$. Suppose that\linebreak ${\mathcal{F}\mathcal{S}\mathcal{S}\lambda_{\mathbb{N}^{'}}[M_{n}, P_{n}]\ <\ \frac{1}{2}}$. Then there exists a natural number ${K \in \mathbb{N}}$ such that for large enough ${n \in \mathbb{N}^{'}}$
\[
\dim M_{n}\ <\ \dim P_{n}\ +\ K.
\]
\end{theorem}
By the method used in the proof of theorem $2.2.2$ from \cite{burshteyn} one can prove the above theorem in the case when $X$ is a Hilbert space by using the fact that any subspace of a Hilbert space allows for a projection on that subspace of norm $1$. It is also not hard to see that the validity of the above theorem may be established if the following statement is true which currently appears to us to be an open question:
\begin{theorem}
Let $X$ be a Banach space and ${(M_{n})_{\mathbb{N}}}$, ${(N_{n})_{\mathbb{N}}}$ be two sequences of finite-dimensional subspaces from $X$ such that ${\dim M_{n} - \dim N_{n}\ \rightarrow\ \infty}$ (note that ${(N_{n})_{\mathbb{N}} \prec (M_{n})_{\mathbb{N}}}$ is not a requirement). Let ${(\varepsilon_{n})_{\mathbb{N}} \rightarrow 0}$ be a sequence of nonnegative real numbers. Then there exists a subsequence of subspaces ${(K_{n})_{\mathbb{N}^{'}} \prec (M_{n})_{\mathbb{N}^{'}}}$ such that ${\dim K_{n} \rightarrow \infty}$ and ${dist (x, N_{n}) \geq 1 - \varepsilon_{n}}$ for every unit vector ${x \in K_{n}}$ for every ${n \in \mathbb{N}^{'}}$.
\end{theorem}
M. I. Ostrovskii had mentioned in a private correspondence that this theorem may be related to the known geometrical results from A. Dvoretzky \cite{dvoretzky}, V.D. Milman \cite{milman0} and A. Pe\l czy\'nski \cite{pelczynski2} and that while the above theorem might not end up being true, its weaker version might be proved if the constant $1$ from the estimate is replaced with the constant ${\frac{1}{2}}$ so that it becomes ${dist (x, N_{n}) \geq \frac{1}{2} - \varepsilon_{n}}$.
%\input{intro} %You need a file 'intro.tex' for this.
%
%
%%%%%%%%%%%%%%%%%%%%%%%%%%%%%%%%%%%%%%%%%%%%%%%%%%%%%%%%%%%%%
%% ==> Some hints are following:
%
%\chapter{Some small hints}\label{hints}
%
%\section{References}\label{references}
%Using the commands \verb#\label{name}# and \verb#\ref{name}# you are able
%to use references in your document. Advantage: You do not need to think
%about numerations, because \LaTeX\ is doing that for you.
%
%For example, in section \ref{dividing} on page \pageref{dividing} hints for
%dividing large documents are given.
%
%Certainly, references do also work for tables, figures, formulas\ldots
%
%Please notice, that \LaTeX\ usually needs more than one run (mostly 2) to
%resolve those references correctly.
%
%
%\section{Dividing large Documents}\label{dividing}
%You can divide your \LaTeX-Document into an arbitrary number of \TeX-Files
%to avoid too big and therefore unhandy files (e.g. one file for every chapter).
%
%For this, you insert in your main file (this one) for every subfile
%the command '\verb#\input{subfile}#'. This leads to the same behavior
%as if the content of the subfile would be at the place of the \verb#\input#-Command.
%
%% <== End of hints
%%%%%%%%%%%%%%%%%%%%%%%%%%%%%%%%%%%%%%%%%%%%%%%%%%%%%%%%%%%%%

%%%%%%%%%%%%%%%%%%%%%%%%%%%%%%%%%%%%%%%%%%%%%%%%%%%%%%%%%%%%%
%% BIBLIOGRAPHY AND OTHER LISTS
%%%%%%%%%%%%%%%%%%%%%%%%%%%%%%%%%%%%%%%%%%%%%%%%%%%%%%%%%%%%%
%% A small distance to the other stuff in the table of contents (toc)
\addtocontents{toc}{\protect\vspace*{\baselineskip}}

%% The Bibliography
%% ==> You need a file 'literature.bib' for this.
%% ==> You need to run BibTeX for this (Project | Properties... | Uses BibTeX)
%\addcontentsline{toc}{chapter}{Bibliography} %'Bibliography' into toc
%\nocite{*} %Even non-cited BibTeX-Entries will be shown.
%\bibliographystyle{alpha} %Style of Bibliography: plain / apalike / amsalpha / ...
%\bibliography{literature} %You need a file 'literature.bib' for this.

\begin{thebibliography}{99}

\bibitem{aiena}
   P.~Aiena,
   \emph{Fredholm and local spectral theory with applications to multipliers},
   Kluwer Academic Publishers,
   (2004).
\bibitem{burshteyn}
   B.\,I.~Burshteyn,
   \emph{Uniform $\lambda-$Adjustment and $\mu-$Approximation in Banach Spaces},
   Arxiv:math.FA/0804.2832 Vol. {\bf 1},
   (17 Apr 2008).
\bibitem{diestel_jaqrchow_pietsch}
   J.~Diestel, H.~Jarchow, A.~Pietsch,
   \emph{Operator ideals},
   Handbook of the Geometry of Banach Spaces,
   Vol. {\bf 1}, (2003), Elsevier, pp.~437--496.
\bibitem{dvoretzky}
   A.~Dvoretzky,
   \emph{Some results on convex bodies and Banach spaces},
   Proc. Int. Symp. on Linear Spaces,
   Jerusalem 1961, pp.~123--160.
\bibitem{gohberg_krein}
   I.\,C.~Gohberg, M.\,G.~Krein,
   \emph{The basic propositions on defect numbers, root numbers, and indices of linear operators},
   Uspehi Mat. Nauk,
   Vol. {\bf 12}, (2), (1957), pp.~43--118.
\bibitem{goldberg_thorp}
   S.~Goldberg, E.~Thorp,
   \emph{On some open questions concerning strictly singular operators},
   Proc. Amer. Math. Soc.,
   Vol. {\bf 14}, (1963), pp.~334--336.
\bibitem{goldberg0}
   S.~Goldberg,
   \emph{Unbounded Linear Operators},
   Mc. Graw-Hill,
   (1966), New York.
\bibitem{gonzales}
   M.~Gonzales,
   \emph{Fredholm theory for pairs of closed subspaces of a Banach space},
   J. Math. Anal. Appl.
   Vol. {\bf 305} (2005), pp.~53--62.
\bibitem{kato0}
   T.~Kato,
   \emph{Perturbation theory for nullity, deficiency, and other quantities of linear operators},
   J. Analyse Math.,
   Vol. {\bf 6}, (1958), pp.~281--322.
\bibitem{kato}
   T.~Kato,
   \emph{Perturbation theory for linear operators},
   Springer-Verlag,
   (1995), pp.~218--236.
\bibitem{krein_krasnoselskii_milman}
   M.\,G.~Krein, M.\,A.~Krasnosel'ski\~{\i}, D.\,P.~Mil'man,
   \emph{On the defect numbers of linear operators in Banach space and on some geometric problems},
   Sbornik Trud. Inst. Mat. Akad. Nauk Ukr. SSR,
   Vol. {\bf 11}, (1948), pp.~97--112.
\bibitem{milman0}
   V.\,D.~Milman,
   \emph{Spektra of bounded continuous functions defined on the unit sphere of a B-space},
   Funkcional. Anal. i Prilozhen.,
   Vol. {\bf 3}, No. 2 (1969), pp.~67--79. 
\bibitem{milman}
   V.\,D.~Milman,
   \emph{Operators of class ${C_{0}}$ and ${C^{*}_{0}}$}, 
   Teor. Funkci\~{\i} Funkcional. Anal. i Prilozhen.,
   No. 10, (1970), pp.~15--26.
\bibitem{ostrovskii}
   M.\,I.~Ostrovskii,
   \emph{Topologies on the set of all subspaces of a Banach space and related questions of Banach space geometry},
   Quaestiones Math.,
   {\bf 17}, (1994), pp.~259--319.
\bibitem{pelczynski}
   A.~Pe\l czy\'nski,
   \emph{On strictly singular and strictly co-singular operators I, II},
   Bull. Akad. Polon. Sci. Ser. Sci. Math. Astronom. Phys.,
   Vol. {\bf 13}, (1965), pp.~31--36,37--41.
\bibitem{pelczynski2}
   A.~Pe\l czy\'nski, 
   \emph{All separable Banach spaces admit for every $\varepsilon >0$ fundamental total and bounded by $1+\varepsilon $ biorthogonal sequences}, 
   Studia Math., 
   Vol. {\bf 55} (1976), no. 3, pp.~295--304.
\bibitem{plichko}
   A.~Plichko,
   \emph{Superstrictly singular and superstrictly cosingular operators in Functional analysis and its applications},
   Elsevier, Amsterdam,
   (2004).
\bibitem{Sari_Schlumprecht_Tomczak_Jaegermann_Troitsky}
   B.~Sari, Th.~Schlumprecht, N.~Tomczak-Jaegermann, V.G.~Troitsky,
   \emph{On norm closed ideals in ${L(l_{p} \oplus l_{q})}$},
   arXive:math.FA/0509414 Vol. {\bf 1},
   19 sep 2005.
\bibitem{schechter}
   M.~Schechter,
   \emph{Quantities related to strictly singular operators},
   Indiana University Math. J.,
   Vol. {\bf 21}, (1972), pp.~1061--71.

\end{thebibliography}

%% The List of Figures
%%\clearpage
%%\addcontentsline{toc}{chapter}{List of Figures}
%%\listoffigures

%% The List of Tables
%%\clearpage
%%\addcontentsline{toc}{chapter}{List of Tables}
%%\listoftables

%%%%%%%%%%%%%%%%%%%%%%%%%%%%%%%%%%%%%%%%%%%%%%%%%%%%%%%%%%%%%
%% APPENDICES
%%%%%%%%%%%%%%%%%%%%%%%%%%%%%%%%%%%%%%%%%%%%%%%%%%%%%%%%%%%%%
\appendix
%% ==> Write your text here or include other files.

%\input{FileName} %You need a file 'FileName.tex' for this.

\end{document}